\documentclass[a4paper]{amsart}

\RequirePackage{amsmath} 
\RequirePackage{amssymb}
\usepackage{amscd,latexsym,amsthm,amsfonts,amssymb,amsmath,amsxtra}
\usepackage[colorlinks=true,urlcolor=blue,citecolor=blue]{hyperref}
\usepackage{color}
\usepackage[all]{xy}
\usepackage[OT2,T1]{fontenc}
\usepackage{bm}
\usepackage{mathtools}
\usepackage{ mathrsfs }
\usepackage{xcolor}
\usepackage{comment}
\usepackage{marginnote}
\usepackage{enumitem}
\usepackage{thmtools, thm-restate}

\DeclareSymbolFont{cyrletters}{OT2}{wncyr}{m}{n}
\DeclareMathSymbol{\Sha}{\mathalpha}{cyrletters}{"58}

\let\Re\undefined

\DeclareMathOperator{\Re}{Re}

\DeclareMathOperator{\Tr}{Tr}

\DeclareMathOperator{\supp}{supp}

\DeclareMathOperator{\Spec}{Spec}

	\newcommand{\Res}{\operatorname{Res}}
	
	\newcommand{\Eis}{\operatorname{Eis}}
	\newcommand{\Geo}{\operatorname{Geo}}
	
	\newcommand{\K}{\operatorname{K}}
	\newcommand{\sgn}{\operatorname{sgn}}
	
	\newcommand{\du}{\operatorname{Dual}}
	
	\newcommand{\Ad}{\operatorname{Ad}}
	
	\newcommand{\Ram}{\operatorname{Ram}}

	\newcommand{\Reg}{\operatorname{Reg}}

	\newcommand{\fin}{\operatorname{fin}}
	\newcommand{\diag}{\operatorname{diag}}

	\newcommand{\Vol}{\operatorname{Vol}}

	\newcommand{\sm}{\operatorname{Small}}
	\newcommand{\bi}{\operatorname{Big}}
	\newcommand{\new}{\operatorname{new}}

	\newcommand{\dist}{\operatorname{dist}}

	\newcommand{\Ind}{\operatorname{Ind}}
	\newcommand{\ER}{\operatorname{ER}}

	\newcommand{\RNum}[1]{\uppercase\expandafter{\romannumeral #1\relax}}

\begin{document}
\theoremstyle{plain}
\newtheorem{thm}{Theorem}[section]
	
\newtheorem{cor}[thm]{Corollary}
\newtheorem{thmy}{Theorem}
\renewcommand{\thethmy}{\Alph{thmy}}
\newenvironment{thmx}{\stepcounter{thm}\begin{thmy}}{\end{thmy}}
\newtheorem{cory}{Corollary}
\renewcommand{\thecory}{\Alph{cory}}
\newenvironment{corx}{\stepcounter{thm}\begin{cory}}{\end{cory}}
\newtheorem{hy}[thm]{Hypothesis}
\newtheorem*{thma}{Theorem A}
\newtheorem*{corb}{Corollary B}
\newtheorem*{thmc}{Theorem C}
\newtheorem{lemma}[thm]{Lemma}  
\newtheorem{prop}[thm]{Proposition}
\newtheorem{conj}[thm]{Conjecture}
\newtheorem{fact}[thm]{Fact}
\newtheorem{claim}[thm]{Claim}
	
\theoremstyle{definition}
\newtheorem{defn}[thm]{Definition}
\newtheorem{example}[thm]{Example}
\theoremstyle{remark}
	
\newtheorem{remark}[thm]{Remark}	
\numberwithin{equation}{section}
	
\title[]{Average of Central $L$-values for $\mathrm{GL}(2)\times\mathrm{GL}(1),$ Hybrid Subconvexity, and Simultaneous Nonvanishing}%
\author{Liyang Yang}
	
\begin{abstract}
We employ a regularized relative trace formula to establish a second moment estimate for twisted $L$-functions across all aspects over a number field. Our results yield hybrid subconvex bounds for both Hecke $L$-functions and twisted $L$-functions, comparable to the Weyl bound in suitable ranges. Moreover, we present an application of our results to address the simultaneous nonvanishing problem.
\end{abstract}
	
\date{\today}%
\maketitle
\tableofcontents

\section{Introduction}

Central $L$-values of modular forms play important roles in number theory and arithmetic geometry. The relative trace formula, introduced in \cite{RR05}, has emerged as a powerful analytic tool for studying the average behavior of central $L$-values for holomorphic cusp forms. Building upon this, \cite{FW09} extended the analysis to include Hilbert modular forms over total real fields. In this article, we employ a regularized relative trace formula to investigate central values of general automorphic $L$-functions for $\mathrm{GL}(2)\times\mathrm{GL}(1)$ over a number field. Our approach yields several new results, including a second moment estimate that encompasses all aspects and incorporates stability concepts from \cite{MR12}, hybrid-type subconvexity bounds for both Hecke $L$-functions and twisted $L$-functions that can rival the strength of the Weyl bound in the appropriate range, and an improved bound on simultaneous nonvanishing in the level aspect.

\subsection{Hybrid Second Moment Involving Stability}
Our first result is the following bound towards the second moment of twisted $L$-functions.
\begin{thmx}\label{A}
Let $F$ be a number field with ring of adeles $\mathbb{A}_F$. Let $\chi$ be a Hecke character of $\mathbb{A}_F^{\times}/F^{\times}$ with arithmetic conductor $Q=C_{\fin}(\chi)$. Let $\mathfrak{M}$ be an integral ideal of norm $M.$ For $v\mid\infty,$ let $c_v, C_v, T_v>0.$ Set $T=\prod_{v\mid\infty}T_v.$ Let $\Pi_{\infty}=\otimes_{v\mid\infty}\Pi_v$  be an irreducible admissible generic representation of $\mathrm{GL}(2)/F_{\infty}.$ Let $\mathcal{A}_0(\Pi_{\infty},\mathfrak{M};\chi_{\infty},\omega)$ be the set of cuspidal automorphic representations $\pi=\otimes_v\pi_v$ of $\mathrm{GL}(2)/F$ with central character $\omega$  such that $\pi_{\fin}=\otimes_{v<\infty}\pi_v$ has arithmetic conductor dividing $\mathfrak{M},$ and $\pi_{v}\otimes\chi_{v}\simeq \Pi_v$ has uniform parameter growth of size $(T_v;c_v,C_v),$ for all $v\mid\infty,$ cf. \textsection\ref{2.1.2}. Then 
\begin{equation}\label{eq1.1}
\sum_{\pi\in\mathcal{A}_0(\Pi_{\infty},\mathfrak{M};\chi_{\infty},\omega)}|L(1/2,\pi\times\chi)|^2\ll (TMQ)^{\varepsilon}(TM+T^{\frac{1}{2}}Q\cdot \textbf{1}_{M\ll Q^2\gcd(M,Q)}),
\end{equation}
where the implied constants depend on $\varepsilon,$ $F,$ $c_v,$ and $C_v,$ $v\mid\infty$.
\end{thmx}

Theorem \ref{A} is an analog and extension of the results in \cite{FW09} from Hilbert modular forms on an anisotropic quaternion algebra to cuspidal automorphic representations of $\mathrm{GL}(2)$ over general number fields. The estimate \eqref{eq1.1} incorporates the explicit dependence on the spectral parameter $T$ by utilizing Nelson's test function at the archimedean places. Notably, there are no restrictions on the arithmetic conductors, allowing $M$ and $Q$ to be arbitrary.


The condition $\textbf{1}_{M\ll Q^2\gcd(M,Q)}$ in \eqref{eq1.1} captures the stability of regular orbital integrals, akin to the treatment in \cite{FW09}, although the specific regular orbital integrals under consideration  differ significantly.

For $F=\mathbb{Q}$, with $\Pi_{\infty}$ being a holomorphic discrete series of $\mathrm{SL}(2)$, and $\chi$ as a Dirichlet character, Theorem \ref{A} implies the following. 

\begin{cor}\label{cor1.5}
Let $k\geq 2$ and $N\geq 1.$ Let $\chi$ be a primitive Dirichlet character modulo $q.$ Then 
\begin{equation}\label{eq1.2}
\sum_{f\in \mathcal{F}_{k}^{\new}(N)}|L(1/2,f\times\chi)|^2\ll (kNq)^{\varepsilon}(kN+k^{\frac{1}{2}}q\cdot \textbf{1}_{N\ll q^2\gcd(N,q)}),
\end{equation}
where the implied constant depends only on $\varepsilon.$ Here $\mathcal{F}_{k}^{new}(N)$ is an orthogonal basis of normalized new forms that are holomorphic Hecke eigenforms with weight $k$ and level $N$, and have trivial nebentypus. 
\end{cor}

Note that \eqref{eq1.2} improves \cite{Kha21} by explicating the dependence on $k$, and allowing for arbitrary values of $N$ and $q$.

\subsection{Hybrid Weyl Subconvex Bounds}
Dropping all but one terms on the left hand side of \eqref{eq1.1} we then obtain the following hybrid bound for twisted $L$-functions. 

\begin{thmx}\label{B}
Let $F$ be a number field with ring of adeles $\mathbb{A}_F$. Let $\pi$ be either a unitary cuspidal automorphic representation of $\mathrm{GL}(2)/F$ or a unitary Eisenstein series.  Let $\chi$ be a Hecke character of $\mathbb{A}_F^{\times}/F^{\times}$. Suppose that $\pi_{v}\otimes\chi_{v}$ has uniform parameter growth of size $(T_v;c_v,C_v),$ for all $v\mid\infty,$ cf. \textsection\ref{2.1.2}. Then 
\begin{equation}\label{eq1.3}
L(1/2,\pi\times\chi)\ll C(\pi\otimes\chi)^{\varepsilon}\big[T^{\frac{1}{2}}C_{\fin}(\pi)^{\frac{1}{2}}+T^{\frac{1}{4}}C_{\fin}(\chi)^{\frac{1}{2}}\big],
\end{equation}
where the implied constant depends on $\varepsilon,$ $F,$ $c_v,$ and $C_v,$ $v\mid\infty$. In particular, 
\begin{equation}\label{eq1.4}
L(1/2,\pi\times\chi)\ll_{\pi_{\infty},\chi_{\infty},F,\varepsilon} C_{\fin}(\pi\times\chi)^{\frac{1}{6}+\varepsilon} 
\end{equation}
if $(C_{\fin}(\pi),C_{\fin}(\chi))=1$ and $C_{\fin}(\pi)^{1-\varepsilon}\ll C_{\fin}(\chi)\ll C_{\fin}(\pi)^{1+\varepsilon}.$
\end{thmx}

When considering a CM extension $E/F$, where $\pi$ corresponds to a Hilbert modular form over $F$ and $\sigma_{\Omega}$ represents the theta series associated with an ideal class group character $\Omega$ of $E$, a hybrid variant of \eqref{eq1.3} for $L(1/2, \pi \times \sigma_{\Omega})$ has been established in \cite[Theorem 1.4]{FW09} through the utilization of a relative trace formula on a quaternion algebra. This relative trace formula, together with a selection of local test function, has further been employed in \cite[Theorem 1.8]{HP18} to derive a hybrid subconvexity outcome in a similar fashion.

In the case of $\mathrm{GL}(2)\times\mathrm{GL}(1)$ over $F=\mathbb{Q}$, the Weyl bound $L(1/2,\pi\times\chi)\ll C_{\fin}(\chi)^{\frac{1}{3}+\varepsilon}$ was established by \cite{CI00} for a fixed cusp form $\pi$ of $\mathrm{PGL}(2)$ and a quadratic Dirichlet character $\chi$. This result was further generalized by  \cite{PY20}, where the Weyl bound $L(1/2,\pi\times\chi)\ll C_{\fin}(\pi\times\chi)^{\frac{1}{6}+\varepsilon}$ is proven under the conditions $\chi^2\neq 1$, $\pi$ has a level dividing $C_{\fin}(\chi)$, and $\pi$ has a central character $\overline{\chi}^2$. In particular, $C_{\fin}(\pi)$ is not coprime to $C_{\fin}(\chi)$. Consequently, \eqref{eq1.4} addresses a complementary case to \cite{PY20}.

By taking $\omega=\eta^2$ for some Hecke character $\eta$ and $\pi=\eta\boxplus\eta,$ we obtain the following bound for Hecke $L$-functions.
\begin{cor}\label{C}
Let $F$ be a number field with ring of adeles $\mathbb{A}_F$. Let $\eta$ and $\chi$ be Hecke character of $\mathbb{A}_F^{\times}/F^{\times}$ with coprime arithmetic conductors. Then 
\begin{equation}\label{1.2.}
L(1/2,\eta\chi)\ll \min\big\{C_{\fin}(\eta)^{\frac{1}{2}+\varepsilon}+C_{\fin}(\chi)^{\frac{1}{4}+\varepsilon},C_{\fin}(\eta)^{\frac{1}{4}+\varepsilon}+C_{\fin}(\chi)^{\frac{1}{2}+\varepsilon}\big\},
\end{equation}
where the implied constant depends on $F,$ $\varepsilon,$ $\eta_{\infty},$ and $\chi_{\infty}.$ In particular, 
\begin{align*}
L(1/2,\chi)\ll_{F,\chi_{\infty},\varepsilon}C_{\fin}(\chi)^{\frac{1}{6}+\varepsilon}
\end{align*}
if $\chi=\chi_1\chi_2$ with $(C_{\fin}(\chi_1),C_{\fin}(\chi_2))=1$ and $C_{\fin}(\chi_1)^{2-\varepsilon}\ll C_{\fin}(\chi_2)\ll C_{\fin}(\chi_1)^{2+\varepsilon}.$ 
\end{cor}



\subsection{Applications to Simultaneous Nonvanishing}

Corollary \ref{cor1.5} serves as a versatile alternative to multiple third moment estimates in certain applications. It replaces Young's third moment bound (cf. \cite[Theorem 1.1]{You17}) in \cite{Luo17} and provides a substantial improvement to the level aspect simultaneous nonvanishing result (cf. \cite[Theorem 1.2]{Kum20}), replacing Petrow-Young's third moment estimate \cite[Theorem 1]{PY19} with the use of Corollary \ref{cor1.5}.

\begin{cor}\label{cornon}
Let $k\in \{2,3,4,5,7\}.$ Let $N\geq 2$ be a prime. Denote by $\mathcal{F}_{2k}^{new}(N)$ an orthogonal basis of normalized new forms that are holomorphic Hecke eigenforms with weight $2k$ and level $N$, and have trivial nebentypus. Let $f\in\mathcal{F}_{2k}^{new}(N).$ Then there exists a nontrivial primitive quadratic character $\chi$ such that
\begin{equation}\label{eq1.6}
\#\big\{g\in \mathcal{F}_{2k}^{\new}(N):\ L(1/2,f\times\chi)L(1/2,g\times\chi)\neq 0\}\gg_{\varepsilon}N^{1-\varepsilon},
\end{equation}
where the implied constant depends on $\varepsilon$.
\end{cor}
\begin{remark}
The lower bound $N^{1-\varepsilon}$ in Corollary \ref{cornon} significantly improves the main result in \cite{Kum20}, where the lower bound achieved was $N^{1/2-\varepsilon}$.
\end{remark}

\subsection{Discussion of the Proofs} 
Let $A=\diag(\mathrm{GL}(1), 1),$ $G=\mathrm{GL}(2)$, and $\overline{G}=\mathrm{PGL}(2)$. Let $f$ be a nice function on $G(\mathbb{A}_F)$. Denote by 
\begin{align*}
\K(g_1,g_2)=\sum_{\gamma\in \overline{G}(F)}f(g_1^{-1}\gamma g_2),\ \ g_1, g_2\in G(\mathbb{A}_F)
\end{align*}
the associated kernel function, which also admits a spectral expansion. By substituting these expansions of $\K(x,y)$ into the integral
\begin{equation}\label{1.4}
\int_{A(F)\backslash A(\mathbb{A}_F)}\int_{A(F)\backslash A(\mathbb{A}_F)}\K(x,y)\chi(x)\overline{\chi}(y)d^{\times}xd^{\times}y,
\end{equation}
we obtain a formal equality between two divergent expressions. To regularize it, we establish an identity between two holomorphic functions on $\mathbb{C}^2$ in the form of
\begin{equation}\label{2.24}
J_{\Spec}^{\Reg,\heartsuit}(f,\textbf{s},\chi)=J_{\Geo}^{\Reg,\heartsuit}(f,\textbf{s},\chi),\quad \textbf{s}\in \mathbb{C}^2,
\end{equation}
where evaluating this identity at $\mathbf{s}=(0,0)$ provides a regularization of \eqref{1.4}.

\subsubsection{The spectral side: a lower bound}
We will prove a lower bound 
\begin{equation}\label{1.8}
J_{\Spec}^{\Reg,\heartsuit}(f,\textbf{0},\chi)\gg T^{-\frac{1}{2}-\varepsilon}(MQ)^{-\varepsilon} \sum_{\pi\in\mathcal{A}_0(\Pi_{\infty},\mathfrak{M};\chi_{\infty},\omega)}|L(1/2,\pi\times\chi)|^2.
\end{equation}

A more comprehensive version that includes the continuous spectrum is given by Theorem \ref{thm6} in \textsection\ref{sec3}.

\subsubsection{The geometric side: an upper bound}
According to types of orbital integrals, we decompose the geometric side into three integrals
\begin{equation}\label{geom}
J_{\Geo}^{\Reg,\heartsuit}(f,\textbf{0},\chi)=J^{\Reg}_{\Geo,\sm}(f,\chi)+J^{\Reg}_{\Geo,\du}(f,\chi)+J^{\Reg,\RNum{2}}_{\Geo,\bi}(f,\textbf{0},\chi).
\end{equation}
\begin{itemize}
	\item The terms $J^{\Reg}_{\Geo,\sm}(f,\chi)$ and $J^{\Reg}_{\Geo,\du}(f,\chi)$ correspond to irregular orbital integrals, exhibiting an asymptotic magnitude of $T^{\frac{1}{2}+o(1)}M^{1+o(1)}.$
	\medskip
	
	\item The term $J^{\Reg,\RNum{2}}_{\Geo,\bi}(f,\textbf{0},\chi)$ represents the contribution from regular orbital integrals, which constitutes the main focus of this paper. We establish that it is bounded by $\ll T^{\varepsilon}M^{\varepsilon}Q^{1+\varepsilon}\cdot \textbf{1}_{M\ll Q^2\gcd(M,Q)}.$ 
\end{itemize}

Based on the above estimates, we obtain an upper bound for the geometric side
\begin{equation}\label{1.10}
J_{\Geo}^{\Reg,\heartsuit}(f,\textbf{0},\chi)\ll T^{\frac{1}{2}+\varepsilon}M^{1+\varepsilon}+T^{\varepsilon}M^{\varepsilon}Q^{1+\varepsilon}\cdot \textbf{1}_{M\ll Q^2\gcd(M,Q)},
\end{equation}

Using equations \eqref{2.24}, \eqref{1.8}, and \eqref{1.10}, we establish Theorem \ref{B} for the case where $\pi$ is cuspidal.

\subsubsection{Some Remarks}
The approach utilized in this work exhibits similarities to that of \cite{Yan23c}, albeit with notable distinctions in the treatment of test functions at ramified places. In \cite{Yan23c}, the focus is primarily on the case of joint ramification, where $Q\mid M$, resulting in relatively simpler regular orbital integrals that can be further improved through nontrivial bounds on specific character sums. However, in the case of totally disjoint ramification, where $(M,Q)=1$, the regular orbital integrals do not exhibit any oscillatory behavior, and the trivial bound becomes optimal. This paper addresses the most general situation, allowing $M$ and $Q$ to take arbitrary values. Another difference from the aforementioned work is that we evaluate the expressions at $\textbf{s}=(0,0)$ (instead of some $\textbf{s}_0=(s_0,s_0)$ with $s_0>0$) in order to compute the second moment over the family. This necessitates careful consideration of singularity matching when computing the main term  $J^{\Reg}_{\Geo,\sm}(f,\chi)+J^{\Reg}_{\Geo,\du}(f,\chi).$     

By employing a straightforward `trivial' estimate of the regular orbital integrals, we establish convexity in the $\chi$-aspect and achieve strong hybrid subconvexity. This represents one of the key advantages of the relative trace formula. The robust nature of this approach holds promise for deriving bounds for higher rank Rankin-Selberg $L$-functions in the level aspect. In future work, we intend to extend the techniques presented in this paper to higher ranks, building upon the general regularized relative trace formula introduced in \cite{Yan23a}.


\subsection{Outline of the Paper}
\subsubsection{The Regularized Relative Trace Formula}
In \textsection \ref{sec2}, we introduce the notations that will be consistently used throughout the paper, along with setting up the local and global data. Additionally, we define the test functions that will play a crucial role in the relative trace formula.

Moving to \textsection \ref{Sec3}, we derive the regularized relative trace formula summarized in Theorem \ref{thm2.4} and Corollary \ref{cor3.3} in \textsection\ref{sec3.7}. 
\subsubsection{The Spectral Side}

In \textsection\ref{sec3}, we explore the spectral side $J_{\Spec}^{\Reg,\heartsuit}(f,\textbf{0},\chi)$. Its meromorphic continuation is obtained in \textsection\ref{sec3.4}. By combining this with the local estimates developed in \textsection\ref{sec3.5}--\textsection\ref{sec3.6}, we establish a lower bound for the spectral side (cf. Theorem \ref{thm6}) in terms of the second moment of central $L$-values.

\subsubsection{The Geometric Side}
In \textsection\ref{sec4}--\textsection\ref{sec6} we handle the geometric side 
\begin{align*}
J_{\Geo}^{\Reg,\heartsuit}(f,\textbf{0},\chi)=J^{\Reg}_{\Geo,\sm}(f,\chi)+J^{\Reg}_{\Geo,\du}(f,\chi)+J^{\Reg,\RNum{2}}_{\Geo,\bi}(f,\textbf{0},\chi).
\end{align*}
\begin{enumerate}
\item The small cell orbital integral $J^{\Reg}_{\Geo,\sm}(f,\chi)$, one of the main terms, is addressed in Proposition \ref{prop12} in \textsection\ref{4.1.4}, utilizing local estimates from \textsection\ref{sec5.1}--\textsection\ref{4.1.3}.
\item The dual orbital integral $J_{\Geo,\du}^{\bi}(f,\chi)$ is bounded by Proposition \ref{prop17} in \textsection\ref{sec5}. This integral is considered `dual' to $J_{\Geo,\sm}^{\bi}(f,\chi)$ through Poisson summation and contributes as the other main term.
\item The regular orbital integrals $J^{\Reg,\RNum{2}}_{\Geo,\bi}(f,\textbf{0},\chi)$ present the most challenging aspect of the geometric side $J_{\Geo}^{\Reg,\heartsuit}(f,\textbf{0},\chi)$. Their behaviors are outlined in Theorem \ref{thmD} in \textsection\ref{sec6}.
\end{enumerate}

\subsubsection{Proof of Main Results}
With the aforementioned preparations, we are able to prove the main results in \textsection\ref{sec7}. In \textsection\ref{sec7.1}--\textsection\ref{sec7.3} we put estimates from the spectral and geometric side all together, obtaining Theorem \ref{thmE}, which yields Theorem \ref{B}.


\subsection{Notation}\label{notation}
\subsubsection{Number Fields and Measures}\label{1.1.1}
Let $F$ be a number field with ring of integers $\mathcal{O}_F.$ Let $N_F$ be the absolute norm. Let $\mathfrak{O}_F$ be the different of $F.$ Let $\mathbb{A}_F$ be the adele group of $F.$ Let $\Sigma_F$ be the set of places of $F.$ Denote by $\Sigma_{F,\fin}$ (resp. $\Sigma_{F,\infty}$) the set of nonarchimedean (resp. archimedean) places. For $v\in \Sigma_F,$ we denote by $F_v$ the corresponding local field and $\mathcal{O}_v$ its ring of integers. For a nonarchimedean place $v,$ let  $\mathfrak{p}_v$ be the maximal prime ideal in $\mathcal{O}_v.$ Given an integral ideal $\mathcal{I},$ we say $v\mid \mathcal{I}$ if $\mathcal{I}\subseteq \mathfrak{p}_v.$ Fix a uniformizer $\varpi_{v}\in\mathfrak{p}_v.$ Denote by $e_v(\cdot)$ the evaluation relative to $\varpi_v$ normalized as $e_v(\varpi_v)=1.$ Let $q_v$ be the cardinality of $\mathcal{O}_v/\mathfrak{p}_v.$ We use $v\mid\infty$ to indicate an archimedean place $v$ and write $v<\infty$ if $v$ is nonarchimedean. Let $|\cdot|_v$ be the norm in $F_v.$ Put $|\cdot|_{\infty}=\prod_{v\mid\infty}|\cdot|_v$ and $|\cdot|_{\fin}=\prod_{v<\infty}|\cdot|_v.$ Let $|\cdot|_{\mathbb{A}_F}=|\cdot|_{\infty}\otimes|\cdot|_{\fin}$. We will simply write $|\cdot|$ for $|\cdot|_{\mathbb{A}_F}$ in calculation over $\mathbb{A}_F^{\times}$ or its quotient by $F^{\times}$.   

Let $\psi_{\mathbb{Q}}$ be the additive character on $\mathbb{Q}\backslash \mathbb{A}_{\mathbb{Q}}$ such that $\psi_{\mathbb{Q}}(t_{\infty})=\exp(2\pi it_{\infty}),$ for $t_{\infty}\in \mathbb{R}\hookrightarrow\mathbb{A}_{\mathbb{Q}}.$ Let $\psi_F=\psi_{\mathbb{Q}}\circ \Tr_F,$ where $\Tr_F$ is the trace map. Then $\psi_F(t)=\prod_{v\in\Sigma_F}\psi_v(t_v)$ for $t=(t_v)_v\in\mathbb{A}_F.$ For $v\in \Sigma_F,$ let $dt_v$ be the additive Haar measure on $F_v,$ self-dual relative to $\psi_v.$ Then $dt=\prod_{v\in\Sigma_F}dt_v$ is the standard Tamagawa measure on $\mathbb{A}_F$. Let $d^{\times}t_v=\zeta_{F_v}(1)dt_v/|t_v|_v,$ where $\zeta_{F_v}(\cdot)$ is the local Dedekind zeta factor. In particular, $\Vol(\mathcal{O}_v^{\times},d^{\times}t_v)=\Vol(\mathcal{O}_v,dt_v)=N_{F_v}(\mathfrak{D}_{F_v})^{-1/2}$ for all finite place $v.$ Moreover, $\Vol(F\backslash\mathbb{A}_F; dt_v)=1$ and $\Vol(F\backslash\mathbb{A}_F^{(1)},d^{\times}t)=\underset{s=1}{\Res}\ \zeta_F(s),$ where $\mathbb{A}_F^{(1)}$ is the subgroup of ideles $\mathbb{A}_F^{\times}$ with norm $1,$ and $\zeta_F(s)=\prod_{v<\infty}\zeta_{F_v}(s)$ is the finite Dedekind zeta function. Denote by $\widehat{F^{\times}\backslash\mathbb{A}_F^{(1)}}$  the Pontryagin dual of $F^{\times}\backslash\mathbb{A}_F^{(1)}.$

\subsubsection{Reductive Groups}
For an algebraic group $H$ over $F$, we will denote by $[H]:=H(F)\backslash H(\mathbb{A}_F).$ We equip measures on $H(\mathbb{A}_F)$ as follows: for each unipotent group $U$ of $H,$ we equip $U(\mathbb{A}_F)$ with the Haar measure such that, $U(F)$ being equipped with the counting measure and the measure of $[U]$ is $1.$ We equip the maximal compact subgroup $K$ of $H(\mathbb{A}_F)$ with the Haar measure such that $K$ has total mass $1.$ When $H$ is split, we also equip the maximal split torus of $H$ with Tamagawa measure induced from that of $\mathbb{A}_F^{\times}.$

In this paper we set  $A=\diag(\mathrm{GL}(1),1),$ and $G=\mathrm{GL}(2).$ Let $B$ be the group of upper triangular matrices in $G$.  
Let $\overline{G}=Z\backslash G$ and $B_0=Z\backslash B,$ where $Z$ is the center of $G.$ Let $T_B$ be the diagonal subgroup of $B$. Then $A\simeq Z\backslash T_B.$ Let $N$ be the unipotent radical of $B$. Let $K=\otimes_vK_v$ be a maximal compact subgroup of $G(\mathbb{A}_F),$ where $K_v=\mathrm{U}_2(\mathbb{C})$ is $v$ is complex, $K_v=\mathrm{O}_2(\mathbb{R})$ if $v$ is real, and $K_v=G(\mathcal{O}_v)$ if $v<\infty.$ For $v\in \Sigma_{F,\fin},$ $m\in\mathbb{Z}_{\geq 0},$ define 
\begin{equation}\label{2.1}
K_v[m]:=\Big\{\begin{pmatrix}
a&b\\
c&d
\end{pmatrix}\in G(\mathcal{O}_v):\ c\in \mathfrak{p}_v^{m}\Big\}.
\end{equation}

\subsubsection{Automorphic Data}
Let $\textbf{s}=(s_1, s_2)\in\mathbb{C}^2.$ 
Let $\omega\in \widehat{F^{\times}\backslash\mathbb{A}_F^{(1)}}.$ Denote by $\mathcal{A}_0\left([G],\omega\right)$ the set of cuspidal representations on $G(\mathbb{A}_F)$ with central character $\omega.$ 

For $\eta_1, \eta_2\in \widehat{F^{\times}\backslash\mathbb{A}_F^{(1)}},$ let $\Ind(\eta_1\otimes\eta_2)$ be the unitary parabolic induction from $B(\mathbb{A}_F)$ to $G(\mathbb{A}_F)$ associated with $\eta_1\otimes\eta_2,$ and let   $\eta_1\boxplus\eta_2$ be Langlands sum. 

\subsubsection{Other Conventions}\label{sec1.5.4}
For a function $h$ on $G(\mathbb{A}_F),$ we define $h^*$ by assigning $h^*(g)=\overline{h({g}^{-1})},$ $g\in G(\mathbb{A}_F).$ Let $F_1(s), F_2(s)$ be two meromorphic functions. Write $F_1(s)\sim F_2(s)$ if there exists an \textit{entire} function $E(s)$ such that $F_1(s)=E(s)F_2(s).$ Denote by $\alpha\asymp \beta$ for $\alpha, \beta \in\mathbb{R}$ if there are absolute constants $c$ and $C$ such that $c\beta\leq \alpha\leq C\beta.$

Throughout the paper, we adhere to the $\varepsilon$-convention, wherein $\varepsilon$ denotes a positive number that can be chosen arbitrarily small, though it may vary between different instances.

\textbf{Acknowledgements}
I am grateful to Dinakar Ramakrishnan for his helpful discussions. I would also like to extend my thanks to Caltech for their warm hospitality during my visit, where this paper was written.

\section{Choice of the Test Function}\label{sec2}
The notations introduced in this section will be extensively utilized throughout the remainder of this paper.
\subsection{Intrinsic Data}\label{sec2.1}
Let $F$ be a number field. Let $\chi=\otimes_v\chi_v$ and $\omega=\otimes_v\omega_v$ be primitive unitary Hecke characters of $F^{\times}\backslash\mathbb{A}_F^{\times}$. Let $\mathfrak{M}$ be an integral ideal of norm $|\mathfrak{M}|:=N_F(\mathfrak{M}).$

\subsubsection{Analytic Conductor of Hecke Characters}\label{2.1.1}
Let $C(\chi):=\otimes_{v\in \Sigma_F}C_v(\chi)$ be the \textit{analytic} conductor of $\chi,$ where each local conductor $C_v(\chi)$ is defined as follows. 
\begin{itemize}
	\item For $F_v\simeq \mathbb{R},$ $\chi_v=\sgn^{n_v'}|\cdot|^{i\kappa_v},$ $n_v'\in\{0,1\},$ we define 
\begin{align*}
C_v(\chi)=1+\Big|\frac{n_v'+i\kappa_v}{2}\Big|.
\end{align*}
\item For $F_v\simeq \mathbb{C},$ and $\chi_v(a)=(a/|a|)^{n_v'}|a|^{2i\kappa_v},$ $a\in F_v^{\times},$ we define
\begin{align*}
C_v(\chi):=(1+|i\kappa_v+|n_v'|/2|)^2.
\end{align*}

\item For $v<\infty,$ let $n_v$ the exponent of $\chi_v,$ namely, $r_{\chi_v}$ is the smallest nonnegative integer such that $\chi_v$ is trivial over $1+\varpi_v^{r_{\chi_v}}\mathcal{O}_v^{\times}$ but not over $1+\varpi_v^{r_{\chi_v}-1}\mathcal{O}_v^{\times}.$ Let $C_v(\chi)=q_v^{r_{\chi_v}}.$ 
\end{itemize}

Denote by $C_{\infty}(\chi):=\otimes_{v\mid\infty}C_v(\chi)$ and $C_{\fin}(\chi):=\otimes_{v<\infty}C_v(\chi).$

\subsubsection{Analytic Conductor of Automorphic Representations of $\mathrm{GL}(2)/F$}
Let $\pi=\otimes_{v}\pi_v$ be an automorphic representation of $G(\mathbb{A}_F)$ with central character $\omega_{\pi}=\omega=\otimes_v\omega_v.$  Let $C(\pi):=\otimes_vC_v(\pi)$ be the analytic conductor of $\pi,$ where each local conductor $C_v(\pi)$ is defined as follows.
\begin{itemize}
	\item Let $v<\infty$. We denote by $r_{\pi_v}\geq 0$ the exponent of $\pi_v$, which is the least integer such that $\pi_v$ has a vector that is $K_{v}[r_{\pi_v}]$-invariant (as defined in \eqref{2.1}). The local conductor of $\pi_v$ is defined as $C_{v}(\pi):=q_v^{r_{\pi_v}}$.
	\item For $v\mid\infty$, the local $L$-function of $\pi_v$ can be expressed as a product of shifted Gamma factors, given by $L_v(s,\pi_v)=\Gamma_v(s+\beta_{1,v})\Gamma_v(s+\beta_{2,v}),$ where $\beta_{1,v}, \beta_{2,v}\in\mathbb{C},$ and $\Gamma_v$ represents the Gamma function over $F_v$. Let
	\begin{align*}
		C_v(\pi):=[(1+|\beta_{1,v}|)(1+|\beta_{2,v}|)]^{[F_v:\mathbb{R}]}.
	\end{align*}  
\end{itemize}  

 Let $C_{\fin}(\pi)=\prod_{v<\infty}C_v(\pi)$ be the arithmetic conductor of $\pi$ and let $C_{\infty}(\pi)=\prod_{v\mid\infty}C_v(\pi)$ be the archimedean conductor of $\pi.$ 
 
\subsubsection{Uniform Parameter Growth}\label{2.1.2}
Let $\Pi_{\infty}=\otimes_{v\mid\infty}\Pi_v$  be an irreducible admissible generic representation of $\mathrm{GL}(2)/F_{\infty}.$  For $v\mid\infty,$ let $L_v(s,\Pi_v)=\Gamma_v(s+\gamma_{1,v})\Gamma_v(s+\gamma_{2,v})$ be the associated $L$-factor of $\Pi_{v}.$ 
\begin{defn}\label{hy}
For $v\mid\infty,$ we say that $\Pi_v$ has \textit{uniform parameter growth of size $(T_v;c_v,C_v)$} for some constants $c_v$ and $C_v,$ and parameters $T_v,$ if $c_vT_v\leq |\gamma_{j,v}|\leq C_vT_v.$
\end{defn}

\subsubsection{Ramification Parameters}\label{2.1.5} 
For $v\in\Sigma_{F,\fin},$ let $e_v(\cdot)$ be the normalized evaluation of $F_v$ such that $e_v(\varpi_v)=1.$ Following the notation in \textsection\ref{2.1.1}, let  $r_{\chi_v}$ (resp. $r_{\omega_v}$) be the exponent of $\chi_v$ (resp. $\omega_v$). We set $m_v:=e_v(\mathfrak{M})$ and $n_v:=r_{\chi_v}.$ Let $\Sigma_{\Ram}^+:=\{v\in \Sigma_{F_{\fin}}:\ m_v\geq n_v\geq 1\},$ and $\Sigma_{\Ram}^-:=\{v\in \Sigma_{F_{\fin}}:\ m_v< n_v,\ n_v\geq 1\}.$ Let $K_v[m_v]$ and $K_v[n_v]$ be defined by \eqref{2.1}.

Denote by $\mathfrak{Q}=\prod_{v<\infty}\mathfrak{p}_v^{n_v}.$ For simplicity we write $Q=C_{\fin}(\chi),$ $M=|\mathfrak{M}|:=N_F(\mathfrak{M}),$ and $M'=C(\omega_{\fin}).$  Suppose that $Q>1.$ Note that $M'\mid M.$

\subsubsection{The Family of Automorphic Forms}\label{auto}
Let $c_v$ and $C_v$ be positive constants for each $v\mid\infty$, and let $T_v>0$. In this paper, we will vary $T_v$ as needed, while keeping $c_v$ and $C_v$ fixed. Let $T=\prod_{v\mid\infty} T_v.$

For $v\mid\infty,$ let $\Pi_v$ be an irreducible admissible generic representation of $\mathrm{GL}(2)/F_v,$ which uniform parameter growth of size $(T_v;c_v,C_v),$ cf. \textsection\ref{2.1.2}.

\begin{itemize}
	\item Let $\mathcal{A}_0(\Pi_{\infty},\mathfrak{M};\chi_{\infty},\omega)$ be the set of cuspidal automorphic representations $\pi=\otimes_v\pi_v$ of $\mathrm{GL}(2)/F$ such that 
\begin{itemize}
	\item $\pi$ has central character $\omega,$ 
	\item for all $v<\infty,$ $\pi_{v}$ has a $K_v[m_v]$-invariant vector, i.e., $r_{\pi_v}\leq e_v(\mathfrak{M}).$ 
	\item $\pi_{v}\otimes\chi_{v}\simeq \Pi_v$ at each $v\mid\infty.$
\end{itemize}

Note that Weyl law yields $\#\mathcal{A}_0(\Pi_{\infty},\mathfrak{M};\chi_{\infty},\omega)=(T|\mathfrak{M}|)^{1+o(1)}.$
\item Let $\mathcal{X}_0(\Pi_{\infty},\mathfrak{M};\chi_{\infty},\omega)$ be the set of Hecke characters $\eta=\otimes_v\eta_v\in \widehat{F^{\times}\backslash\mathbb{A}_F^{(1)}}$ such that 
\begin{itemize}
	\item for all $v<\infty,$ the representation $\eta_{v}\boxplus \omega_v\overline{\eta}_v$ has a $K_v[m_v]$-invariant vector, i.e., $r_{\eta_v}+r_{\omega_v\overline{\eta}_v}\leq m_v,$
	\item $\eta_{v}\chi_v\boxplus \omega_v\overline{\eta}_v\chi_v\simeq \Pi_v$ at each $v\mid\infty.$
\end{itemize}

\end{itemize}

\subsubsection{Other Notations}
For a function $h$ on $G(\mathbb{A}_F)$ or $G(F_v),$ $v\in\Sigma_F,$ define $h^*(g)=\overline{h(g^{-1})}$ and 
\begin{equation}\label{3..}
(h*h^*)(g)=\int h(gg'^{-1})h^*(g')dg'=\int h(gg')\overline{h(g')}dg'.
\end{equation}

\subsection{Construction of Test Functions}\label{3.2}

We construct a test function $f$ on $G(\mathbb{A}_F)$ using the following procedure:
\begin{itemize}
	\item For the archimedean places (cf. \textsection \ref{3.2.1}), we rely on Nelson's work \cite{Nel20} (cf. \textsection1.5.2 and \textsection14 on p.80) and follow the approach described in \cite{NV21}, \textsection1.10. Additional information can be found in \cite{Nel21}, Part 2. 
	
	\item For the finite places, we employ the test function constructed in \cite[\textsection 2.2]{Yan23c}, which involves a double average over unipotent translations weighted by characters (cf. \textsection\ref{2.2.2}).
\end{itemize}

\subsubsection{Construction of ${f}_{\infty}$}\label{3.2.1} 
Let $v\mid\infty.$ Recall that $\Pi_v$ has \textit{uniform parameter growth of size $(T_v;c_v,C_v)$} (cf. Definition \ref{hy} in \textsection\ref{2.1.2}). Then $\Pi_v$ has \textit{uniform parameter growth of size $(T_v;c_v/2,2C_v)$}, where $s_0$ is the parameter defined by \eqref{eq2.1} in \textsection\ref{2.1.5.}.

Let $\mathfrak{g}$ (resp. $\mathfrak{g}'$) be the Lie algebras of $G(F_v)$ (resp. $A(F_v)$),  with imaginal dual $\hat{\mathfrak{g}}$ (resp. $\hat{\mathfrak{g}}'$). One can choose an element $\tau\in\hat{\mathfrak{g}}$ with the restriction $\tau'=\tau\mid_{A}\in \hat{\mathfrak{g}}',$ so that $\tau$ (resp. $\tau'$) lies in the coadjoint orbit $\mathcal{O}_{\Pi_v}$ of $\Pi_v$ (resp. $\mathcal{O}_{\textbf{1}_v}$ of $\textbf{1}_v$ the trivial representation of $A(F_v)$). Let $\tilde{f}^{\wedge}_{v}:$ $\hat{\mathfrak{g}}\rightarrow\mathbb{C}$ be a smooth bump function concentrated on $\{\tau+(\xi,\xi^{\bot}):\ \xi\ll T_v^{\frac{1}{2}+\varepsilon},\ \xi^{\bot}\ll T_v^{\varepsilon}\},$ where $\xi$ lies in the tangent space of $\mathcal{O}_{\Pi_v}$ at $\tau,$ and $\xi^{\bot}$ has the normal direction. Let $\tilde{f}_{v}\in C_c^{\infty}(G(F_v))$ be the pushforward of the Fourier transform of $\tilde{f}_{v}^{\wedge}$ truncated at the essentially support, namely,  
\begin{equation}\label{245}
	\supp \tilde{f}_{v}\subseteq \big\{g\in G(F_v):\ g=I_{n+1}+O(T_v^{-\varepsilon}),\ \Ad^*(g)\tau=\tau+O(T_v^{-\frac{1}{2}+\varepsilon})\big\},
\end{equation}
where the implied constants rely on $c_v$ and $C_v.$ 

Then, in the sense of \cite{NV21}, \textsection {2.5}, the operator $\pi_{v}(\tilde{f}_{v})$ is approximately a rank one projector with range spanned by a unit vector microlocalized at $\tau.$ Let
\begin{equation}\label{3.}
f_{v}(g):=f_v(g,\chi_v)*f_v(g,\chi_v)^*,
\end{equation} 
where $v\mid\infty,$ $g\in G(F_v),$ and 
\begin{equation}\label{6.}
f_v(g,\chi_v):=\chi_v(\det g)\int_{Z(F_v)}\tilde{f}_{v}(zg)\omega_{v}(z)d^{\times}z.
\end{equation}

Due to the support of $\tilde{f}$, the function $f_v(g)$ is non-zero unless $|\det g|_v > 0$. Therefore, $\sgn(|\det g|_v) = 1$. As a result, the function $f_v$ is smooth on $G(F _v)$.

\subsubsection{Application of Transversality}\label{2.2.2.}
By definition, one has (cf. (14.13) in \cite{Nel21}) 
\begin{equation}\label{250}
	\|\tilde{f}_{v}\|_{\infty}\ll_{\varepsilon} T_v^{1+\varepsilon}, \ \ v\mid\infty,
\end{equation}
where $\|\cdot \|_{\infty}$ is the sup-norm. For $g\in\overline{G}(F_v),$ we may write 
\begin{align*}
g=\begin{pmatrix}
a&b\\
c&d
\end{pmatrix}\in G(F_v),\ \ \ g^{-1}=\begin{pmatrix}
a'&b'\\ 
c'&d'
\end{pmatrix}\in G(F_v).
\end{align*}
Define 
\begin{equation}\label{dg}	
d_{v}(g):=\begin{cases}
\min\big\{1, |d^{-1}b|_v+ |d^{-1}c|_v+ |d'^{-1}b'|_v+|d'^{-1}c'|_v \big\},&\ \text{if $dd'\neq 0,$}\\
1,&\ \text{if $dd'=0$}.
\end{cases}
\end{equation}
\begin{prop}[Theorem 15.1 of \cite{Nel20}]\label{prop3.1}
Let notation be as above. Then there is a fixed neighborhood $\mathcal{Z}$ of the identity in $A(F_v)$ with the following property. Let $g$ be in a small neighborhood of $I_{n+1}$ in $\overline{G}(F_v).$ Let $\delta_v>0$ be small. Then 
\begin{equation}\label{2.7..}
\Vol\left(\big\{z\in\mathcal{Z}:\ \dist(gz\tau, A(F_v)\tau)\leq \delta_v\big\}\right)\ll \frac{\delta_v}{d_v(g)}. 
\end{equation} 
Here $\dist(\cdots)$ denotes the infimum over $g'\in A(F_v)$ of $\|gz\tau-g'\tau\|,$ where $\|\cdot\|$ is a fixed norm on $\hat{\mathfrak{g}}.$
\end{prop}

Proposition \ref{prop3.1} (with $\delta_v=T_v^{-1/2+\varepsilon}$) will be used to detect the restriction $\Ad^*(g)\tau=\tau+O(T_v^{-\frac{1}{2}+\varepsilon})$ in the support of $\tilde{f}_v.$ By \eqref{245}, \eqref{250}, and \eqref{2.7..}, 
\begin{equation}\label{2.7}
|\tilde{f}_v(g)|\ll T^{1+\varepsilon}\cdot\textbf{1}_{|\Ad^*(g)\tau-\tau|\ll T_v^{-1/2+\varepsilon}}\cdot\textbf{1}_{|g-I_2|\ll T_v^{-\varepsilon}}\cdot \Big\{1,\frac{T_v^{-1/2+\varepsilon}}{d_v(g)}\Big\}.
\end{equation}

\subsubsection{Finite Places}\label{2.2.2}
For $v\in\Sigma_{F,\fin},$ we define a function on $G(F_v),$ supported on $Z(F_v)\backslash K_{v}[m_v],$ by  
\begin{equation}\label{5.}
f_{v}(z_vk_v;\omega_v)={\Vol(\overline{K_{v}[m_v]})^{-1}}\omega_v(z_v)^{-1}\omega_v(E_{2,2}(k_v))^{-1},
\end{equation}
where $\overline{K_v[m_v]}$ is the image of $K_v[m_v]$ in $\overline{G}(F_v),$ and $E_{2,2}(k_v)$ is the $(2,2)$-th entry of $k_v\in K_v[m_v].$ For $g_v\in G(F_v),$ define by  
\begin{align*}
f_v(g_v)=\frac{1}{|\tau(\chi_v)|^{2}}\sum_{\alpha\in (\mathcal{O}_v/\varpi_v^{n_v}\mathcal{O}_v)^{\times}}\sum_{\beta\in (\mathcal{O}_v/\varpi_v^{n_v}\mathcal{O}_v)^{\times}}\chi_v(\alpha)\overline{\chi}_v(\beta)f_{v}\left(g_{\alpha,\beta,v};\omega_v\right),
\end{align*}
where 
$$
\tau(\chi_v)=\sum_{\alpha\in (\mathcal{O}_v/\varpi_v^{n_v}\mathcal{O}_v)^{\times}}\psi_v(\alpha\varpi_v^{-n_v})\chi_v(\alpha)
$$ 
is the Gauss sum relative to the additive character $\psi_v$, and 
\begin{align*}
	g_{\alpha,\beta,v}:=\begin{pmatrix}
		1&\alpha \varpi_v^{-n_v}\\
		&1
	\end{pmatrix}g_v\begin{pmatrix}
		1&\beta \varpi_v^{-n_v}\\
		&1
	\end{pmatrix}. 
\end{align*}
 
Note that $n_v=0$ for almost all $v\in\Sigma_{F,\fin}.$ Hence, for all but finitely many $v\in\Sigma_{F,\fin},$ the test function $f_v(\cdot)=f_v(\cdot;\omega_v)$ (cf. \eqref{5.}) supports in $Z(F_v)\backslash K_{v}[m_v].$

\subsubsection{Construction of the Test Function}\label{3.2.5}
Let $f=\otimes_{v\in\Sigma_F}f_v,$ where $f_v$ is constructed in \textsection\ref{3.2.1} and \textsection\ref{2.2.2}. Note that $f_{\infty}$ is determined by $\Pi_{\infty}.$ 

\section{The Regularized Relative Trace Formula}\label{Sec3}

\subsection{Fourier Expansion of the Kernel Function}\label{2.3}
Let $f=\otimes_vf_v$ be defined in \textsection\ref{3.2.5}. Then $f$ defines an integral operator 
\begin{equation}\label{p}
	R(f)\phi(g)=\int_{\overline{G}(\mathbb{A}_F)}f(g')\phi(gg')dg'
\end{equation}
on the space $L^2\left([G],\omega\right)$ of functions on $[G]$ which transform under $Z(\mathbb{A}_F)$ by $\omega$ and are square integrable on $[\overline{G}].$ This operator is represented by the kernel function
\begin{equation}\label{11}
	\K(g_1,g_2)=\sum_{\gamma\in \overline{G}(F)}f(g_1^{-1}\gamma g_2),\ \ g_1, g_2\in G(\mathbb{A}_F).
\end{equation}

It is well known that $L^2\left([G],\omega\right)$ decomposes into the direct sum of the space $L_0^2\left([G],\omega\right)$ of cusp forms and spaces $L_{\Eis}^2\left([G],\omega\right)$ and $L_{\Res}^2\left([G],\omega\right)$ defined using Eisenstein series and residues of Eisenstein series respectively. Then
\begin{equation}\label{ker}
	\K_0(g_1,g_2)+\K_{\ER}(g_1,g_2)=\K(g_1,g_2)=\sum_{\gamma\in \overline{G}(F)}f(g_1^{-1}\gamma g_2),
\end{equation}
where $\K_{0}(g_1,g_2)$ (resp. $\K_{\ER}(g_1,g_2)$) is the contribution from the cuspidal (resp. non-cuspidal) spectrum. Explicit expansions of $\K_0(g_1,g_2)$ and $\K_{\ER}(g_1,g_2)$ will be given in \textsection\ref{sec3.1}. 

By Bruhat decomposition $\K(g_1,g_2)=\K_{\sm}(g_1,g_2)+\K_{\bi}(g_1, g_2),$ where 
	\begin{align*}
		\K_{\sm}(g_1,g_2)=\sum_{\gamma\in B_0(F)}f(g_1^{-1}\gamma g_2),\quad \K_{\bi}(g_1,g_2)=\sum_{\gamma\in B_0(F)wN(F)}f(g_1^{-1}\gamma g_2).
	\end{align*}
Let $\mathcal{K}(\cdot,\cdot)\in \{\K(\cdot,\cdot), \K_0(\cdot,\cdot), \K_{\ER}(\cdot,\cdot), \K_{\sm}(\cdot,\cdot),\K_{\bi}(\cdot,\cdot)\}.$ Define 
\begin{align*} 
\mathcal{F}_0\mathcal{F}_1\mathcal{K}(g_1,g_2):=&\int_{[N]}\mathcal{K}(g_1,u_2g_2)du_2,\ \ \mathcal{F}_1\mathcal{F}_0\mathcal{K}(g_1,g_2):=\int_{[N]}\mathcal{K}(u_1g_1,g_2)du_1,\\
\mathcal{F}_1\mathcal{F}_1\mathcal{K}(g_1,g_2):=&\int_{[N]}\int_{[N]}\mathcal{K}(u_1g_1,u_2g_2)du_2du_1,\\
\mathcal{F}_2\mathcal{F}_2\K(g_1,g_2):=&\sum_{\alpha\in A(F)}\sum_{\beta\in A(F)}\int_{[N]}\int_{[N]}\K(u_1\alpha g_1,u_2\beta g_2)\theta(u_1)\overline{\theta}(u_2)du_2du_1.
\end{align*}
Using Poisson summation twice the integral $\mathcal{F}_2\mathcal{F}_2\mathcal{K}(g_1,g_2)$ is equal to 
\begin{equation}\label{8}
\mathcal{K}(g_1,g_2)-\mathcal{F}_0\mathcal{F}_1\mathcal{K}(g_1,g_2)-\mathcal{F}_1\mathcal{F}_0\mathcal{K}(g_1,g_2)+\mathcal{F}_1\mathcal{F}_1\mathcal{K}(g_1,g_2).
\end{equation}

By \cite[Lemma 2.3]{Yan23c} we have, for $x, y\in A(\mathbb{A}_F),$ that 
\begin{align*}
\mathcal{F}_0\mathcal{F}_1\K_{\bi}(x,y)=\mathcal{F}_1\mathcal{F}_0\K_{\bi}(x,y)=\mathcal{F}_1\mathcal{F}_1\K_{\bi}(x,y)\equiv 0.
\end{align*}
Along with \eqref{8} we then obtain that 
\begin{equation}\label{3.5}
\mathcal{F}_2\mathcal{F}_2\K_{\bi}(x,y)=\K_{\bi}(x,y).
\end{equation}

\begin{remark}
Note that \eqref{3.5} only holds over $(x,y)\in A(\mathbb{A}_F)\times A(\mathbb{A}_F).$
\end{remark}

\subsection{The Relative Trace Formula}\label{sec3.7}
\subsubsection{The Spectral Side}
Let $\Re(s_1)\gg 1$ and $\Re(s_2)\gg 1.$ Define
\begin{equation}\label{11}
J_{\Spec}^{\Reg}(f,\textbf{s},\chi):=J_{0}^{\Reg}(f,\textbf{s},\chi)+J_{\Eis}^{\Reg}(f,\textbf{s},\chi),
\end{equation}
the spectral side, where $\textbf{s}=(s_1, s_2)\in\mathbb{C}^2,$ and \begin{align*}
J_{0}^{\Reg}(f,\textbf{s},\chi):=&\int_{[A]}\int_{[A]}\K_0(x,y)|\det x|^{s_1}|\det y|^{s_2}\chi(x)\overline{\chi}(y)d^{\times}xd^{\times}y,\\
J_{\Eis}^{\Reg}(f,\textbf{s},\chi):=&\int_{[A]}\int_{[A]}\mathcal{F}_2\mathcal{F}_2\K_{\ER}(x,y)|\det x|^{s_1}|\det y|^{s_2}\chi(x)\overline{\chi}(y)d^{\times}xd^{\times}y.
\end{align*}

By Proposition 6.4 in \cite{Yan23a} (cf. \textsection 6.2), the integral $J_{\Spec}^{\Reg}(f,\textbf{s},\chi)$ converges absolutely in $\Re(s_1), \Re(s_2)\gg 1.$ In addition, $J_{\Spec}^{\Reg}(f,\textbf{s},\chi)$ admits a holomorphic continuation $J_{\Spec}^{\Reg,\heartsuit}(f,\textbf{s},\chi)$ to $\mathbf{s}\in\mathbb{C}^2.$ We will see in \textsection\ref{sec3} that  $J_{\Spec}^{\Reg,\heartsuit}(f,\textbf{s},\chi)$ is roughly an average of $L(1/2+s_1,\pi\times\chi)L(1/2+s_2,\widetilde{\pi}\times\overline{\chi})$ as $\pi$ varies over families of  unitary automorphic representations of $\mathrm{GL}(2)/F.$ 

\subsubsection{The Geometric Side}\label{sec3.2.2}
By \eqref{3.5} and the decomposition $\K(x,y)=\K_{\sm}(x,y)+\K_{\bi}(x,y),$ the geometric side is 
\begin{equation}\label{2.17}
J_{\Geo}^{\Reg}(f,\textbf{s},\chi):=J^{\Reg}_{\Geo,\sm}(f,\textbf{s},\chi)+J^{\Reg}_{\Geo,\bi}(f,\textbf{s},\chi),
\end{equation}
where $\Re(s_1)\gg 1,$ $\Re(s_2)\gg 1,$ and 
\begin{align*}
J^{\Reg}_{\Geo,\sm}(f,\textbf{s},\chi):=&\int_{[A]}\int_{[A]}\mathcal{F}_2\mathcal{F}_2\K_{\sm}(x,y)|\det x|^{s_1}|\det y|^{s_2}\chi(x)\overline{\chi}(y)d^{\times}xd^{\times}y
,\\
J^{\Reg}_{\Geo,\bi}(f,\textbf{s},\chi):=&\int_{[A]}\int_{[A]}\K_{\bi}(x,y)|\det x|^{s_1}|\det y|^{s_2}\chi(x)\overline{\chi}(y)d^{\times}xd^{\times}y.
\end{align*}

As in \cite[\textsection 5]{Yan23a}, we have
\begin{equation}\label{15.}
J^{\Reg}_{\Geo,\bi}(f,\textbf{s},\chi)=J^{\Reg}_{\Geo,\du}(f,\textbf{s},\chi)+J^{\Reg,\RNum{2}}_{\Geo,\bi}(f,\textbf{s},\chi),	
\end{equation}
where $J_{\Geo,\du}^{\bi}(f,\textbf{s},\chi)$ is defined by 
\begin{equation}\label{3.9}
\int_{\mathbb{A}_F^{\times}}\int_{\mathbb{A}_F^{\times}}f\left(\begin{pmatrix}
	1\\
	x&1
\end{pmatrix}\begin{pmatrix}
	y\\
	&1
\end{pmatrix}\right)|x|^{s_1+s_2}|y|^{s_2}\overline{\chi}(y)d^{\times}yd^{\times}x,
\end{equation}
and the regular orbital $J^{\Reg,\RNum{2}}_{\Geo,\bi}(f,\textbf{s},\chi)$ is defined by 
\begin{equation}\label{17...} 
\sum_{t\in F-\{0,1\}}\int_{\mathbb{A}_F^{\times}}\int_{\mathbb{A}_F^{\times}}f\left(\begin{pmatrix}
	y&x^{-1}t\\
	xy&1
\end{pmatrix}\right)|x|^{s_1+s_2}|y|^{s_2}\overline{\chi}(y)d^{\times}yd^{\times}x.
\end{equation}
Note that \eqref{3.9} converges absolutely in $\Re(s_1+s_2)>1,$ and by \cite[Theorem 5.6]{Yan23a} the integral $J^{\Reg,\RNum{2}}_{\Geo,\bi}(f,\textbf{s},\chi)$ converges absolutely in $\mathbf{s}\in\mathbb{C}^2,$ and in particular, the sum over $t\in F-\{0,1\}$ is \textit{finite}, which is called \textit{stability} of the regular orbital integrals (cf. \cite{FW09}, \cite{MR12}, \cite[\textsection 6]{Yan23c}). Therefore, the geometric side $J_{\Geo}^{\Reg}(f,\textbf{s},\chi)$ admits a holomorphic  continuation $J_{\Geo}^{\Reg,\heartsuit}(f,\textbf{s},\chi)$ to $\mathbf{s}\in\mathbb{C}^2.$ We shall investigate it in \textsection \ref{sec4}-\textsection\ref{sec6}.


\subsubsection{The Regularized Relative Trace Formula}\label{rref}
Note that $\K_0(x,y)=\mathcal{F}_2\mathcal{F}_2\K_0(x,y),$ $\K_0(x,y)+\K_{\ER}(x,y)=\K(x,y)=\K_{\sm}(x,y)+\K_{\bi}(x,y).$ Then by \eqref{3.5},
\begin{align*}
\K_0(x,y)+\mathcal{F}_2\mathcal{F}_2\K_{\ER}(x,y)=\mathcal{F}_2\mathcal{F}_2\K(x,y)=\mathcal{F}_2\mathcal{F}_2\K_{\sm}(x,y)+\K_{\bi}(x,y).
\end{align*}

As a consequence, when $\Re(s_1)\gg 1$ and $\Re(s_2)\gg 1,$
\begin{equation}\label{3.11}
J_{\Spec}^{\Reg}(f,\textbf{s},\chi)=J_{\Geo}^{\Reg}(f,\textbf{s},\chi).
\end{equation}

By applying the singularity matching process described in \cite[\textsection 7]{Yan23a}, the equality \eqref{3.11} extends to its holomorphic continuation, leading to the following equality between two holomorphic functions:

\begin{thm}[The Regularized \textbf{RTF}]\label{thm2.4}
Let notation be as before. Then 
\begin{equation}
J_{\Spec}^{\Reg,\heartsuit}(f,\textbf{s},\chi)=J_{\Geo}^{\Reg,\heartsuit}(f,\textbf{s},\chi)\tag{\ref{2.24}}.
\end{equation}
\end{thm}

In this paper, our focus is on evaluating the above regularized \textbf{RTF} at $\mathbf{s} = \mathbf{0} = (0,0)$. Write $\mathbf{s}'=(s,0).$ Define the following normalized integrals
\begin{align*}
J^{\Reg}_{\Geo,\sm}(f,\chi):=&\Big[J^{\Reg}_{\Geo,\sm}(f,\textbf{s}',\chi)-s^{-1}\underset{s=0}{\Res}\ J^{\Reg}_{\Geo,\sm}(f,\textbf{s},\chi)\Big]_{s=0},\\
J^{\Reg}_{\Geo,\du}(f,\chi):=&\Big[J^{\Reg}_{\Geo,\du}(f,\textbf{s}',\chi)-s^{-1}\underset{s=0}{\Res}\ J^{\Reg}_{\Geo,\du}(f,\textbf{s},\chi)\Big]_{s=0}.
\end{align*}

Notice that $\underset{s=0}{\Res}\ J^{\Reg}_{\Geo,\sm}(f,\textbf{s},\chi)+\underset{s=0}{\Res}\ J^{\Reg}_{\Geo,\du}(f,\textbf{s},\chi)\equiv 0.$ Therefore,
\begin{equation}
J_{\Geo}^{\Reg,\heartsuit}(f,\textbf{0},\chi)=J^{\Reg}_{\Geo,\sm}(f,\chi)+J^{\Reg}_{\Geo,\du}(f,\chi)+J^{\Reg,\RNum{2}}_{\Geo,\bi}(f,\textbf{0},\chi)\tag{\ref{geom}}. 
\end{equation}

\begin{cor}\label{cor3.3}
Let notation be as before. Then 
\begin{align*}
J_{\Spec}^{\Reg,\heartsuit}(f,\textbf{0},\chi)=J^{\Reg}_{\Geo,\sm}(f,\chi)+J^{\Reg}_{\Geo,\du}(f,\chi)+J^{\Reg,\RNum{2}}_{\Geo,\bi}(f,\textbf{0},\chi).
\end{align*}
\end{cor}

\section{The Spectral Side: Meromorphic Continuation and Bounds}\label{sec3}
In this section we shall show that $J_{\Spec}^{\Reg,\heartsuit}(f,\textbf{s},\chi)$ admits a holomorphic continuation to $\mathbf{s}\in\mathbb{C}^2.$ Moreover, we derive a lower bound of it as follows. 
\begin{restatable}[]{thm}{thmf}\label{thm6}
Let notation be as in \textsection\ref{sec2}. Then 
\begin{align*}
J_{\Spec}^{\Reg,\heartsuit}(f,\textbf{0},\chi)\gg\ &T^{-\frac{1}{2}-\varepsilon}(MQ)^{-\varepsilon} \sum_{\pi\in\mathcal{A}_0(\Pi_{\infty},\mathfrak{M};\chi_{\infty},\omega)}|L(1/2,\pi\times\chi)|^2\\
&+T^{-\frac{1}{2}-\varepsilon}(MQ)^{-\varepsilon} \sum_{\eta}\int_{\mathbb{R}}\frac{|L(1/2+it,\eta\chi)L(1/2+it,\omega\overline{\eta}\chi)|^2}{|L(1+2it,\omega\overline{\eta}^2)|^2}dt,
\end{align*}
where $\eta\in \mathcal{X}_0(\Pi_{\infty},\mathfrak{M};\chi_{\infty},\omega),$ and the implied constant depends only on $F,$ $\varepsilon,$ $c_v$ and $C_v$ at $v\mid\infty$.  
\end{restatable}

\subsection{Spectral Side: Meromorphic Continuation}\label{sec3.4}
\subsubsection{Spectral Expansion of the kernel functions}\label{sec3.1}
Let notation be as in \textsection\ref{sec2}. Let $f=\otimes_vf_v$ be defined in \textsection \ref{3.2.5}.   Let $\K_0(x,y)$ and $\K_{\ER}(x,y)$ be defined by \eqref{ker} in \textsection \ref{2.3}. Then by the spectral decomposition we have (e.g., cf. \cite{Art79}) 
\begin{align*}
&\K_0(x,y)=\sum_{\sigma\in \mathcal{A}_0([G],\omega)}\sum_{\phi\in\mathfrak{B}_{\sigma}}\sigma(f)\phi(x)\overline{\phi(y)},\\
&\K_{\ER}(x,y)=\frac{1}{4\pi}\sum_{\eta\in \widehat{F^{\times}\backslash\mathbb{A}_F^{(1)}}}\int_{i\mathbb{R}}\sum_{\phi\in \mathfrak{B}_{\sigma_{0,\eta}}}E(x,\mathcal{I}(\lambda,f)\phi,\lambda)\overline{E(y,\phi,\lambda)}d\lambda.
\end{align*}
Here, $\mathfrak{B}_{\sigma}$ denotes an orthonormal basis of the cuspidal representation $\sigma$, and $\sigma_{0,\eta}$ is given by $\sigma_{0,\eta}=\Ind(\eta,\eta^{-1}\omega)$.

\subsubsection{Rankin-Selberg Periods}\label{sec3.3}
Let $\theta=\otimes_v\theta_v$ be the generic induced by the fixed additive character $\psi$ (cf. \textsection\ref{1.1.1}). For a generic automorphic form $\varphi$  on $G(\mathbb{A}_F),$ define the associated Whittaker function by 
\begin{equation}\label{whi}
W_{\varphi}(g):=\int_{[N]}\varphi(ug)\overline{\theta(u)}du, \ \ g\in G(\mathbb{A}_F),
\end{equation}
Using the multiplicity one property, we can express $W_{\varphi}(g)$ as a product over all places $v\in\Sigma_F$ as $W_{\varphi}(g)=\prod_{v\in\Sigma_F}W_{\varphi,v}(g_v)$, where $g=\otimes_vg_v\in G(\mathbb{A}_F)$. The local Whittaker function $W_{\varphi,v}$ is spherical for all but finitely many places $v\in\Sigma_F$. Define
\begin{align*}
\Psi(s,\varphi,\chi):=\int_{\mathbb{A}_F^{\times}}W_{\varphi}\left(\begin{pmatrix}
	x\\
	&1
\end{pmatrix}\right)|x|^s\chi(x)d^{\times}x=\prod_{v\in\Sigma_F}\Psi_v(s,\varphi,\chi),
\end{align*}
where the local integral is defined by 
\begin{align*}
\Psi_v(s,\varphi,\chi)=\int_{F_v^{\times}}W_{\varphi,v}\left(\begin{pmatrix}
	x_v\\
	&1
\end{pmatrix}\right)|x_v|_v^s\chi_v(x_v)d^{\times}x_v.
\end{align*}

The integral $\Psi(s,\varphi,\chi)$ converges absolutely in $\Re(s)>1.$ Furthermore, it is related to $L$-functions as follows. 
\begin{itemize}
	\item If $\varphi\in\mathfrak{B}_{\sigma},$ where $\sigma\in \mathcal{A}_0([G],\omega),$ then $\Psi(s,\varphi,\chi)$ converges absolutely for all $s\in\mathbb{C}$, making it an entire function. By Hecke's theory, $\Psi(s,\varphi,\chi)$ serves as the integral representation for the complete $L$-function $\Lambda(s+1/2,\sigma)$.

\item If $\varphi\in\mathfrak{B}_{0,\eta}$ associated with some $\eta\in\widehat{F^{\times}\backslash\mathbb{A}_F^{(1)}},$ then as established in \cite{Jac09}, the function $\Psi(s,\varphi,\chi)$ converges absolutely  in the region $\Re(s)_1\gg 1$ and $\Re(s_2)\gg 1$, and it has a meromorphic continuation to $s\in\mathbb{C}$, representing the complete $L$-function $\Lambda(s+1/2,\eta)\Lambda(s+1/2,\eta^{-1}\omega)$.
\end{itemize}

\subsubsection{Meromorphic Continuation}
According to the construction of the test function $f,$ the Eisenstein series $E(x,\mathcal{I}(\lambda,f)\phi,\lambda),$ $\phi\in\mathfrak{B}_{0,\eta},$ vanishes unless $\phi$ is right invariant under $K_v[m_v],$ where $m_v=e_v(\mathfrak{M}),$ cf.  \textsection\ref{2.1.5}.

Substituting the Rankin-Selberg periods (cf. \textsection\ref{sec3.3}) into the decomposition \eqref{11} we then obtain $J_{\Spec}^{\Reg}(f,\textbf{s},\chi)=J_{0}^{\Reg}(f,\textbf{s},\chi)+J_{\Eis}^{\Reg}(f,\textbf{s},\chi),$ where 
\begin{align*}
J_{0}^{\Reg}(f,\textbf{s},\chi)=&\sum_{\sigma\in\mathcal{A}_0(\Pi_{\infty},\mathfrak{M};\chi_{\infty},\omega)}\sum_{\phi\in\mathfrak{B}_{\sigma}}\Psi(s_1,\sigma(f)\phi)\Psi(s_2,\overline{\phi},\overline{\chi}),\\
J_{\Eis}^{\Reg}(f,\textbf{s},\chi)=&\frac{1}{4\pi}\sum_{\eta}\int_{i\mathbb{R}}\sum_{\phi\in \mathfrak{B}_{\sigma_{0,\eta}}}\Psi(s_1,\sigma_{0,\eta}(f)E(\cdot,\phi,\lambda),\chi)\Psi(s_2,\overline{E(\cdot, \phi,\lambda)},\overline{\chi})d\lambda,
\end{align*}
where $\eta$ ranges through $\in \widehat{F^{\times}\backslash\mathbb{A}_F^{(1)}},$ $\Re(s_1)\gg 1$ and $\Re(s_2)\gg 1.$

The function $J_{0}^{\Reg}(f,\textbf{s},\chi)$ continues to a holomorphic function $J_{0}^{\Reg,\heartsuit}(f,\textbf{s},\chi)$ in $\mathbb{C}^2.$ It is proved in \cite[Proposition 3.2]{Yan23c} that $J_{\Eis}^{\Reg}(f,\textbf{s},\chi)$ extends to a holomorphic function $J_{\Eis}^{\Reg,\heartsuit}(f,\textbf{s},\chi)$ in $-1/4<\Re(s_1), \Re(s_2)<1/4$ with
\begin{align*}
J_{\Eis}^{\Reg,\heartsuit}(f,\textbf{s},\chi)=\frac{1}{4\pi}\sum_{\eta}\int_{i\mathbb{R}}\sum_{\phi\in \mathfrak{B}_{\sigma_{0,\eta}}}\Psi(s_1,\sigma_{0,\eta}(f)E(\cdot,\phi,\lambda),\chi)\Psi(s_2,\overline{E(\cdot, \phi,\lambda)},\overline{\chi})d\lambda,
\end{align*}
where $\eta\in \widehat{F^{\times}\backslash\mathbb{A}_F^{(1)}},$ and the integrand $\Psi(s_1,\sigma_{0,\eta}(f)E(\cdot,\phi,\lambda),\chi)\Psi(s_2,\overline{E(\cdot, \phi,\lambda)},\overline{\chi})$ is identified with its meromorphic continuation. In particular, $J_{\Spec}^{\Reg}(f,\textbf{s},\chi)$ is holomorphic in the region $-1/4<\Re(s_1), \Re(s_2)<1/4.$

\subsection{Spectral Side: the Second Moment}\label{sec3.5}
Let notation be as in \textsection \ref{sec2}. Denote by ${f}^{\ddagger}(g)=\otimes_{v\mid\infty}f_v(g_v,\chi_v)\otimes\otimes_{v\in \Sigma_{F,\fin}} f_v(g_{v};\omega_v),$ where $f_v(\cdot;\omega_v)$ is defined by \eqref{5.}, i.e., $f_v(z_vk_v;\omega_v)={\Vol(\overline{K_{v}[m_v]})^{-1}}\omega_v(z_v)^{-1}\omega_v(E_{2,2}(k_v))^{-1}.$ Define  
\begin{align*}
\varphi^{\ddagger}(x):=\int_{G(\mathbb{A})}{f}^{\ddagger}(g)\prod_{v\mid\mathfrak{Q}}\Bigg[\frac{1}{\tau(\chi_v)}\sum_{\beta}\chi_v(\beta)\sigma_v\left(\begin{pmatrix}
	1&\beta\varpi_v^{-n_v}\\
	&1
\end{pmatrix}\right)\Bigg]\sigma(g)\varphi(x)dg,
\end{align*}
where $\beta_v$ ranges over $\in (\mathcal{O}_v/\varpi_v^{n_v}\mathcal{O}_v)^{\times}.$

To simplify notations, we shall still write $\Psi(s,\phi^{\ddagger},\chi)$ and $\Psi(s,E(\cdot,\phi,\lambda)^{\ddagger},\chi)$ for their holomorphic continuations, respectively. It follows from the construction of $f$ that $J_{0}^{\Reg,\heartsuit}(f,\textbf{0},\chi)$ and $J_{\Eis}^{\Reg,\heartsuit}(f,\textbf{0},\chi)$ can be written as follows.  
\begin{lemma}\label{lem8}
Let notation be as before. Then \begin{align*}
J_{0}^{\Reg,\heartsuit}(f,\textbf{0},\chi)=&\sum_{\sigma\in\mathcal{A}_0([G],\omega)}\sum_{\phi\in\mathfrak{B}_{\sigma}}\big|\Psi(0,\phi^{\ddagger},\chi)\big|^2,\\
J_{\Eis}^{\Reg,\heartsuit}(f,\textbf{0},\chi)=&\frac{1}{4\pi}\sum_{\eta\in \widehat{F^{\times}\backslash\mathbb{A}_F^{(1)}}}\int_{i\mathbb{R}}\sum_{\phi\in \mathfrak{B}_{\sigma_{0,\eta}}}\big|\Psi(0,E(\cdot,\phi,\lambda)^{\ddagger},\chi)\big|^2d\lambda. 
\end{align*}
\end{lemma}

\subsubsection{Local calculations}
The non-archimedean calculation presented in \cite[\textsection 3.5.2]{Yan23c} is as follows:
\begin{lemma}\label{lem9}
Let notation be as before. Let $\sigma\in\mathcal{A}_0([G],\omega).$ Let $\phi\in\sigma$ be a pure tensor. Suppose that $\phi^{\ddagger}\neq 0.$ Then for $v\in\Sigma_{F,\fin},$ we have 
\begin{equation}\label{15..}
\Psi_v(s,\phi^{\ddagger},\chi)=W_{\phi,v}(I_2)L_v(s+1/2,\sigma_v\times\chi_v),\ \ \Re(s)\geq 0. 
\end{equation}
\end{lemma}

\begin{lemma}\label{lem10.}
Let notation be as before. Let $\phi\in\sigma_{\lambda,\eta}$ be a pure tensor. Let $\varphi=E(\cdot,\phi,\lambda).$ Suppose that $\varphi^{\ddagger}\neq 0.$ Then for $v\in\Sigma_{F,\fin},$ $\Re(s)\geq 0,$ we have 
\begin{equation}\label{15.}
\Psi_v(s,\varphi^{\ddagger},\chi)=W_{\varphi,v}(I_2)L_v(s+1/2+\lambda,\eta_v\chi_v)L_v(s+1/2-\lambda,\eta_v^{-1}\chi_v\omega_v). 
\end{equation}
\end{lemma}

\subsection{Spectral Side: the lower bound}\label{sec3.6}
In this section we prove Theorem \ref{thm6}. 

Denote by $f_v^{\circ}:=\int_{Z(F_v)}\tilde{f}_{v}(zg)\omega_{v}(z)d^{\times}z.$ Let $\pi=\pi_{\infty}\otimes\pi_{\fin}$ be a unitary automorphic representation of $\mathrm{GL}(2)/F$ with $\pi_{\infty}\otimes\chi_{\infty}\simeq \Pi_{\infty}.$ Let $v\mid\infty,$ by the properties of $f_v$ (cf. e.g., \cite{Nel21}), 
\begin{equation}\label{37.}
T_v^{-1/4-\varepsilon}\ll_{\varepsilon}\int_{F_v^{\times}}(\pi_v(f_v^{\circ})(W_v\otimes \chi_v))\left(\begin{pmatrix}
	x_v\\
	&1
\end{pmatrix}
\right)d^{\times}x_v\ll_{\varepsilon} T_v^{-1/4+\varepsilon}
\end{equation}
for some $W_v$ in the Kirillov model of $\pi_v.$ By definition \eqref{6.} in \textsection \ref{3.2.1}, we have 
\begin{align*}
\pi_v(f_v(\cdot,\chi_v))W_v\left(\begin{pmatrix}
	x_v\\
	&1
\end{pmatrix}
\right)\chi_v(x_v)=(\pi_v(f_v^{\circ})(W_v\otimes \chi_v))\left(\begin{pmatrix}
	x_v\\
	&1
\end{pmatrix}
\right).
\end{align*}

Hence, $\Psi_v(s_0,\pi_v(f_v(\cdot,\chi_v))W_v,\chi_v)\gg_{\varepsilon} T_v^{-1/4-\varepsilon}$ for some $W_v$ in the Kirillov model of $\pi_v.$ Let $\phi\in\pi$ be a cusp form with Petersson norm $\langle\phi,\phi\rangle=1,$ and Whittaker function $W_{\phi}=\otimes_vW_{\phi,v}$ (defined by \eqref{whi}), such that $W_{\phi,v}=W_v,$ for all $v\mid\infty,$ and $W_{\phi,v}$ is $\prod_{v<\infty}K_v[n_v]$-invariant. Then 
\begin{align*}
\Psi_{v}(0,\phi^{\ddagger},\chi)=\Psi_v(0,\pi_v(f_v(\cdot,\chi_v))W_v,\chi_v)\gg_{\varepsilon} T_v^{-1/4-\varepsilon},
\end{align*}
where the implied constant depends on $\varepsilon,$ $c_v$ and $C_v$ at $v\mid\infty.$

Together with Lemmas \ref{lem8}, \ref{lem9}, and the bound $|W_{\fin}(I_2)|\gg (TM)^{-\varepsilon}$ (cf. \cite{HL94}), 
\begin{align*}
J_{0}^{\Reg,\heartsuit}(f,\textbf{0},\chi)\gg T^{-\frac{1}{2}-\varepsilon}(MQ)^{-\varepsilon} \sum_{\pi\in\mathcal{A}_0(\Pi_{\infty},\mathfrak{M};\chi_{\infty},\omega)}|L(1/2,\pi\times\chi)|^2,
\end{align*}
and $J_{\Eis}^{\Reg,\heartsuit}(f,\textbf{0},\chi)$ is 
\begin{align*}
\gg T^{-\frac{1}{2}-\varepsilon}(MQ)^{-\varepsilon} \sum_{\eta}\int_{t\in\mathbb{R}}\frac{|L(1/2+it,\eta\chi)L(1/2+it,\omega\overline{\eta}\chi)|^2}{|L(1+2it,\omega\overline{\eta}^2)|^2}dt,
\end{align*}
where $\eta\in \mathcal{X}_0(\Pi_{\infty},\mathfrak{M};\chi_{\infty},\omega).$ Therefore, Theorem \ref{thm6} follows.

\section{The Geometric Side: the Orbital Integral $J^{\Reg}_{\Geo,\sm}(f,\chi)$}\label{sec4}
Let $f=\otimes_vf_v$ be text function constructed in  \textsection\ref{3.2}. Let $\mathbf{s}=(s,0)$ with $s\in\mathbb{C}.$ Recall the definition in \textsection\ref{rref}: 
\begin{align*}
J^{\Reg}_{\Geo,\sm}(f,\chi):=\Big[J^{\Reg}_{\Geo,\sm}(f,\textbf{s},\chi)-\underset{s=0}{\Res}\ J^{\Reg}_{\Geo,\sm}(f,\textbf{s},\chi)\Big]_{s=0},
\end{align*}
where, for $\Re(s)>1,$ the small cell orbital integral is defined by (cf. \textsection\ref{sec3.2.2}): 
\begin{align*}
J^{\Reg}_{\Geo,\sm}(f,\textbf{s},\chi):=\int_{\mathbb{A}_F^{\times}}\int_{\mathbb{A}_F^{\times}}f\left(\begin{pmatrix}
	y&b\\
	&1
\end{pmatrix}\right)\psi(xb)|x|^{1+s}\overline{\chi}(y)d^{\times}xd^{\times}y,
\end{align*}
which is a Tate integral representing $\Lambda(1+s,\textbf{1}_F).$

\begin{prop}\label{prop12}
Let notation be as before. Then 
\begin{align*}
J^{\Reg}_{\Geo,\sm}(f,\chi)\ll_{\varepsilon} M^{1+\varepsilon}T^{\frac{1}{2}+\varepsilon},
\end{align*} 
where the implied constant depends only on $F,$ $\varepsilon,$ and $c_v,$ $C_v$ at $v\mid\infty,$ cf. \textsection \ref{2.1.2}.
\end{prop}

\subsection{Local calculations at nonarchimedean places}\label{sec5.1}
Let $\Re(s)>0.$ Let
\begin{align*}
J_{\sm,v}(s):=\int_{F_v^{\times}}|x_v|_v^{1+2s}\int_{F_v^{\times}}\int_{F_v}f_v\left(\begin{pmatrix}
	y_v&b_v\\
	&1
\end{pmatrix}\right)\psi_v(x_vb_v)\overline{\chi}_v(y_v)db_vd^{\times}y_vd^{\times}x_v.
\end{align*}
Take $\mathbf{s}=(s,0).$ By definition, we have
\begin{align*}
J^{\Reg}_{\Geo,\sm}(f,\textbf{s},\chi):=&\prod_{v\in\Sigma_{F}}J_{\sm,v}(s),\ \ \Re(s)>0.
\end{align*}

By \cite[Lemma 4.2]{Yan23c} we have, for $v\in\Sigma_{F,\fin},$ that 
\begin{equation}\label{19}
J_{\sm,v}(s)=\begin{cases}
N_{F_v}(\mathfrak{D}_{F_v})^{-1/2}\Vol(K_v[m_v])^{-1}\textbf{1}_{(\mathfrak{D}_{F_v}^{-1})^{\times}}(x_v),\ & \text{if $v\mid\mathfrak{Q},$}\\
|x_v|_v^{1+2s_0}N_{F_v}(\mathfrak{D}_{F_v})^{-1/2}\Vol(K_v[m_v])^{-1}\textbf{1}_{\mathfrak{D}_{F_v}^{-1}}(x_v),\ & \text{if $v\nmid\mathfrak{Q},$}
\end{cases}
\end{equation}
where $m_v=e_v(\mathfrak{M})$ is defined in \textsection\ref{2.1.5}. Hence, 
\begin{equation}\label{5.2}
J^{\Reg}_{\Geo,\sm}(f,\textbf{s},\chi)=V_F\cdot N_{F}(\mathfrak{D}_{F})^{\frac{1}{2}+s} \zeta_F(1+2s)\prod_{v\mid\infty}J_{\sm,v}(s),
\end{equation}
where $V_F:=\prod_{v<\infty}\Vol(K_v[m_v])^{-1}\asymp |\mathfrak{M}|^{1+o(1)}.$

\subsection{Local estimates at archimedean places}\label{4.1.3}
\begin{lemma}\label{lem5.2}
Let notation be as before. Let $\varepsilon>0$ be a constant. Let $\mathcal{C}:=\{s\in\mathbb{C}:\ |s|=\varepsilon\}$ be the circle of radius $\varepsilon.$ Then for $s\in\mathcal{C},$ we have
\begin{equation}\label{5.3}
\prod_{v\mid\infty}J_{\sm,v}(s)\ll T^{\frac{1}{2}+\varepsilon},
\end{equation}
where the implied constant depends on $F,$ $\varepsilon,$ $c_v$ and $C_v,$ $v\mid\infty.$ 
\end{lemma}
\begin{proof}
Let $s\in\mathcal{C}.$ Denote by 
\begin{align*}
\mathcal{I}_v(x_v,s):=|x_v|_v^{1+2s}\int_{F_v^{\times}}\int_{F_v}f_v\left(\begin{pmatrix}
	y_v&b_v\\
	&1
\end{pmatrix}\right)\psi_v(x_vb_v)\overline{\chi}_v(y_v)db_vd^{\times}y_v.
\end{align*}

By the construction of $f$ we have $\mathcal{I}_v(x_v,s)\neq 0$ unless $x_vT_v^{-1}-\gamma_v\ll_{\varepsilon} T_v^{-1/2+\varepsilon},$ where $\gamma_v$ is determined by $\tau\in\hat{\mathfrak{g}}$, cf. \textsection\ref{3.2.1}. Moreover, by decaying of Fourier transform of $f_v,$ $f_v\left(\begin{pmatrix}
	y_v&b_v\\
	&1
\end{pmatrix}\right)\ll T_v^{-\infty}$ if $|b_v|_v\gg T_v^{-1/2+\varepsilon}.$ Together with \eqref{250},
\begin{equation}\label{45}
\mathcal{I}_v(x_v,s)\ll |x_v|_v^{1+2\varepsilon}\textbf{1}_{x_vT_v^{-1}-\gamma_v\ll_{\varepsilon} T_v^{-1/2+\varepsilon}}\cdot T_{v}^{1+\varepsilon}\cdot T_v^{-1/2+\varepsilon}\cdot T_v^{-1/2+\varepsilon}+O(T_v^{-\infty}),
\end{equation}
where the factor $T_{v}^{1+\varepsilon}$ comes from the sup-norm estimate (cf. \eqref{250}), the first factor $T_v^{-1/2+\varepsilon}$ comes from the range of $y_v$ according to the support of $f_v$ (cf. \eqref{245}), and the second $T_v^{-1/2+\varepsilon}$ comes from the essential range of $b_v,$ i.e., $|b_v|_v\ll T_v^{-1/2+\varepsilon}.$ In particular, the implied constant in \eqref{45} depends only on $F_v,$ $\varepsilon,$ and $c_v,$ $C_v$ at $v\mid\infty.$ 

As a consequence, we have
\begin{align*}
J_{\sm,v}(s)=\int_{F_v^{\times}}\mathcal{I}_v(x_v,s)d^{\times}x_v\ll T_{v}^{\varepsilon}\int_{F_v}|x_v|_v^{2\varepsilon}\textbf{1}_{x_vT_v^{-1}-\gamma_v\ll_{\varepsilon} T_v^{-\frac{1}{2}+\varepsilon}}dx_v+O(T_v^{-\infty}),
\end{align*}
which is $\ll T_v^{1/2+\varepsilon}.$ Then \eqref{5.3} follows.
\end{proof}

\subsection{Proof of Proposition \ref{prop12}}\label{4.1.4}
Let $\varepsilon>0$ be a constant. Let $\mathcal{C}:=\{s\in\mathbb{C}:\ |s|=\varepsilon\}$ be the circle of radius $\varepsilon$ (cf. Lemma \ref{lem5.2}). By Cauchy formula, 
\begin{align*}
J^{\Reg}_{\Geo,\sm}(f,\chi)=\frac{1}{2\pi i}\int_{\mathcal{C}}\frac{J^{\Reg}_{\Geo,\sm}(f,\textbf{s},\chi)}{s}ds
\end{align*}

Substituting \eqref{5.2} into the above integral,
\begin{align*}
J^{\Reg}_{\Geo,\sm}(f,\chi)=\frac{V_F}{2\pi i}\int_{\mathcal{C}}\frac{N_{F}(\mathfrak{D}_{F})^{\frac{1}{2}+s} \zeta_F(1+2s)\prod_{v\mid\infty}J_{\sm,v}(s)}{s}ds.
\end{align*}

By Lemma \ref{lem5.2} we have 
\begin{align*}
J^{\Reg}_{\Geo,\sm}(f,\chi)\ll V_F\int_{\mathcal{C}}\frac{N_{F}(\mathfrak{D}_{F})^{1/2+\varepsilon} T^{1/2+\varepsilon}}{|\varepsilon|}\cdot\max_{s\in\mathcal{C}}|\zeta_F(1+2s)|ds\ll M^{1+\varepsilon}T^{\frac{1}{2}+\varepsilon}.
\end{align*}
Hence, Proposition \ref{prop12} follows. 

\section{The Geometric Side: the Orbital Integral $J_{\Geo,\du}^{\bi}(f,\chi)$}\label{sec5}
Let $f=\otimes_vf_v$ be text function constructed in  \textsection\ref{3.2}. Let $\mathbf{s}=(s,0)$ with $s\in\mathbb{C}.$ Recall the definition in \textsection\ref{rref}: 
\begin{align*}
J^{\Reg}_{\Geo,\du}(f,\chi):=\Big[J^{\Reg}_{\Geo,\du}(f,\textbf{s},\chi)-\underset{s=0}{\Res}\ J^{\Reg}_{\Geo,\du}(f,\textbf{s},\chi)\Big]_{s=0},
\end{align*}
where, for $\Re(s)>0,$ the dual  orbital integral is defined by (cf. \textsection\ref{sec3.2.2}): 
\begin{align*}
J_{\Geo,\du}^{\bi}(f,\textbf{s},\chi):=\int_{\mathbb{A}_F^{\times}}\int_{\mathbb{A}_F^{\times}}f\left(\begin{pmatrix}
	1\\
	x&1
\end{pmatrix}\begin{pmatrix}
	y\\
	&1
\end{pmatrix}\right)|x|^{s}\overline{\chi}(y)d^{\times}yd^{\times}x,
\end{align*}
which is a Tate integral representing $\Lambda(2s,\textbf{1}_F).$ Let $s<0.$ By Poisson summation (or equivalently, the functional equation), 
$J^{\Reg}_{\Geo,\du}(f,\textbf{s},\chi)$ becomes
\begin{align*}
\int_{\mathbb{A}_F^{\times}}\int_{\mathbb{A}_F^{\times}}\int_{\mathbb{A}_F}f\left(\begin{pmatrix}
	1\\
	b&1
\end{pmatrix}\begin{pmatrix}
	y\\
	&1
\end{pmatrix}\right)\psi(bx)|x|^{1-s}\overline{\chi}(y)dbd^{\times}yd^{\times}x.
\end{align*}
 
\begin{prop}\label{prop17}
Let notation be as before. Then 
\begin{equation}\label{41}
J^{\Reg}_{\Geo,\du}(f,\chi)\ll M^{1+\varepsilon}T^{\frac{1}{2}+\varepsilon},
\end{equation}
where the implied constant depends only on $F,$ $\varepsilon,$ and $c_v,$ $C_v$ at $v\mid\infty,$ cf. \textsection \ref{2.1.2}.
\end{prop}

Write
\begin{align*}
J^{\Reg}_{\Geo,\du}(f,\textbf{s},\chi)=\prod_{v\in\Sigma_F}J_{\du,v}(s),
\end{align*}
where $\Re(s)<0,$ and each $J_{\du,v}(s)$ is defined by 
\begin{align*}
\int_{F_v^{\times}}\int_{F_v^{\times}}\int_{F_v}f_v\left(\begin{pmatrix}
	1\\
	b_v&1
\end{pmatrix}\begin{pmatrix}
	y_v\\
	&1
\end{pmatrix}\right)\psi_v(b_vx_v)|x_v|_v^{1-s}\overline{\chi}_v(y_v)db_vd^{\times}y_vd^{\times}x_v.
\end{align*}

Similar to \eqref{19} and \eqref{5.2} we have
\begin{align*}
J^{\Reg}_{\Geo,\du}(f,\textbf{s},\chi)=V_F\cdot N_{F}^{(\mathfrak{Q})}(\mathfrak{D}_{F})^{\frac{1}{2}+s} \zeta_F^{(\mathfrak{Q})}(1+2s)\prod_{v\mid\mathfrak{Q}}J_{\du,v}(s)\prod_{v\mid\infty}J_{\du,v}(s),
\end{align*}
where $V_F:=\prod_{v<\infty}\Vol(K_v[m_v])^{-1}\asymp |\mathfrak{M}|^{1+o(1)},$ $$N_{F}^{(\mathfrak{Q})}(\mathfrak{D}_{F})=\prod_{\substack{\mathfrak{p}\mid \mathfrak{D}_F\\ \mathfrak{p}+\mathfrak{Q}=\mathcal{O}_F}}N_F(\mathfrak{p}),\ \ \zeta_F^{(\mathfrak{Q})}(1+2s):=\prod_{\substack{\mathfrak{p}\ \text{prime}\\
\mathfrak{p}+\mathfrak{Q}=\mathcal{O}_F}}\frac{1}{1-N_F(\mathfrak{p})^{-1-2s}}.
$$

\subsection{Local estimates at ramified places}
At $v\mid\mathfrak{Q},$ by definition 
$$
f_v\left(\begin{pmatrix}
	1\\
	b_v&1
\end{pmatrix}\begin{pmatrix}
	y_v\\
	&1
\end{pmatrix}\right)=0
$$ 
unless  
\begin{align*}
z_v\begin{pmatrix}
	1&\alpha \varpi_v^{-n_v}\\
	&1
\end{pmatrix}\begin{pmatrix}
	1\\
	b_v&1
\end{pmatrix}\begin{pmatrix}
	y_v\\
	&1
\end{pmatrix}\begin{pmatrix}
	1&\beta \varpi_v^{-n_v}\\
	&1
\end{pmatrix}\in K_v[m_v]
\end{align*}
for some $\alpha,\beta\in (\mathcal{O}_v/\varpi_v^{n_v}\mathcal{O}_v)^{\times},$ and $z_v\in F_v^{\times},$ i.e., 
\begin{equation}\label{36}
z_v\begin{pmatrix}
	y_v+\alpha b_vy_v\varpi_v^{-n_v}&(y_v+\alpha b_vy_v\varpi_v^{-n_v})\beta \varpi_v^{-n_v}+\alpha \varpi_v^{-n_v}\\
	b_vy_v&1+\beta b_vy_v\varpi_v^{-n_v}
\end{pmatrix}\in K_v[m_v].
\end{equation}

Analyzing the $(2,1)$-th entry of the matrix on the LHS of \eqref{36} yields that $e_v(z_v)+e_v(b_v)+e_v(y_v)\geq m_v\geq n_v.$ Hence an investigation of the $(1,1)$-th and $(2,2)$-th entry leads to  
\begin{align*}
\begin{cases}
e_v(y_v)+2e_v(z_v)=0\\
e_v(z_v)+e_v(y_v)\geq 0,\ \ e_v(z_v)\geq 0.
\end{cases}
\end{align*}
As a consequence, $e_v(z_v)=0,$ i.e., $z_v\in\mathcal{O}_v^{\times}.$ So $e_v(y_v)=0,$ $e_v(b_v)\geq m_v.$ 
 
Hence we have $f_v\left(\begin{pmatrix}
	1\\
	b_v&1
\end{pmatrix}\begin{pmatrix}
	y_v\\
	&1
\end{pmatrix}\right)=\textbf{1}_{\mathcal{O}_v^{\times}}(y_v)\textbf{1}_{\varpi_v^{m_v}\mathcal{O}_v}(b_v).$ After a change of variable (i.e., $\beta\mapsto y_v^{-1}\beta$),
\begin{align*}
\mathcal{J}_v(x_v)=\frac{|\tau(\chi_v)|^{-2}}{\Vol(K_v[m_v])}\sum_{\alpha,\beta}\int_{\varpi_v^{m_v}\mathcal{O}_v} \textbf{1}_{K_v[m_v]}\left(X_v\right)\psi_v(b_vx_v)\chi_v(\alpha)\overline{\chi}_v(\beta)db_v,
\end{align*}
where $X_v$ denotes the matrix
\begin{align*}
\begin{pmatrix}
	1+\alpha b_v\varpi_v^{-n_v}&(1+\alpha b_v\varpi_v^{-n_v})\beta \varpi_v^{-n_v}+\alpha \varpi_v^{-n_v}\\
	b_v&1+\beta b_v\varpi_v^{-n_v}
\end{pmatrix}.
\end{align*}

Note that $\textbf{1}_{K_v[m_v]}\left(X_v\right)\neq 0$ unless $(1+\alpha b_v\varpi_v^{-n_v})\beta +\alpha\in \varpi_v^{n_v}\mathcal{O}_v.$ Hence, 
\begin{align*}
\mathcal{J}_v(x_v)=\frac{|\tau(\chi_v)|^{-2}\chi_v(-1)}{\Vol(K_v[m_v])}\sum_{\alpha\in (\mathcal{O}_v/\varpi_v^{n_v}\mathcal{O}_v)^{\times}}\int_{\varpi_v^{m_v}\mathcal{O}_v} \psi_v(b_vx_v)\chi_v(1+\alpha b_v\varpi_v^{-n_v})db_v.
\end{align*}

Write $b_v=\varpi_v^m\gamma_v,$ $\gamma_v\in\mathcal{O}_v^{\times}.$ Changing the variable $\alpha\mapsto \gamma_v^{-1}\alpha,$  
\begin{align*}
\mathcal{J}_v(x_v)=\frac{|\tau(\chi_v)|^{-2}\chi_v(-1)}{\Vol(K_v[m_v])\zeta_{F_v}(1)}\sum_{m\geq m_v}q_v^{-m}G(m)R(m,x_v).
\end{align*}
where $G$ is the character sum
\begin{align*}
G(m):=\sum_{\alpha\in (\mathcal{O}_v/\varpi_v^{n_v}\mathcal{O}_v)^{\times}}\chi_v(1+\alpha \varpi_v^{m-n_v}),
\end{align*}
and $R(m,x_v)$ is the Ramanujan sum 
\begin{align*}
R(m,x_v):=\int_{\mathcal{O}_v^{\times}} \psi_v(\gamma_v\varpi_v^mx_v)d^{\times}\gamma_v.
\end{align*}

Applying the trivial bound $G(m)\ll q_v^{n_v},$ and $R(m,x_v)=0$ if $m<-e_v(x_v)-1,$ $R(m,x_v)\ll q_v^{-1}$ if $m=-e_v(x_v)-1,$ and $R(m,x_v)=1$ if $m\geq -e_v(x_v),$ we then deduce that 
\begin{equation}\label{6.3.}
J_{\du,v}(s)\ll (m_v+2n_v)\Vol(K_v[m_v])^{-1}.
\end{equation}

\subsection{Local estimates at archimedean places}
Similar to Lemma \ref{lem5.2} we have  
\begin{lemma}\label{lem6.2.}
Let notation be as before. Let $\varepsilon>0$ be a constant. Let $\mathcal{C}:=\{s\in\mathbb{C}:\ |s|=\varepsilon\}$ be the circle of radius $\varepsilon.$ Then for $s\in\mathcal{C},$ we have
\begin{align*}
\prod_{v\mid\infty}J_{\du,v}(s)\ll T^{\frac{1}{2}+\varepsilon},
\end{align*}
where the implied constant depends on $F,$ $\varepsilon,$ $c_v$ and $C_v,$ $v\mid\infty.$ 
\end{lemma}

\subsection{Proof of Proposition \ref{prop17}} 
Let $\varepsilon>0$ be a constant. Let $\mathcal{C}:=\{s\in\mathbb{C}:\ |s|=\varepsilon\}$ be the circle of radius $\varepsilon$ (cf. Lemma \ref{lem6.2.}). By Cauchy formula, 
\begin{align*}
J^{\Reg}_{\Geo,\du}(f,\chi)=\frac{1}{2\pi i}\int_{\mathcal{C}}\frac{J^{\Reg}_{\Geo,\du}(f,\textbf{s},\chi)}{s}ds
\end{align*}

Plugging the expression of $J^{\Reg}_{\Geo,\du}(f,\textbf{s},\chi)$ into the above integral,
\begin{align*}
J^{\Reg}_{\Geo,\sm}(f,\chi)=\frac{V_F}{2\pi i}\int_{\mathcal{C}}\frac{N_{F}^{(\mathfrak{Q})}(\mathfrak{D}_{F})^{\frac{1}{2}+s} \zeta_F^{(\mathfrak{Q})}(1+2s)\prod_{v\mid\mathfrak{Q}}J_{\du,v}(s)\prod_{v\mid\infty}J_{\du,v}(s)}{s}ds.
\end{align*}

By the estimate \eqref{6.3.} and Lemma \ref{lem6.2.} we have 
\begin{align*}
J^{\Reg}_{\Geo,\du}(f,\chi)\ll V_F^{1+\varepsilon}\int_{\mathcal{C}}\frac{N_{F}(\mathfrak{D}_{F})^{\frac{1}{2}+\varepsilon} T^{\frac{1}{2}+\varepsilon}}{|\varepsilon|}\cdot\max_{s\in\mathcal{C}}|\zeta_F^{(\mathfrak{Q})}(1+2s)|ds\ll M^{1+\varepsilon}T^{\frac{1}{2}+\varepsilon}.
\end{align*}
Hence, Proposition \ref{prop17} follows.

\section{The Geometric Side: Regular Orbital Integrals}\label{sec6}
Recall the definition \eqref{17...} in \textsection\ref{sec3.7}: 
\begin{align*}
J^{\Reg,\RNum{2}}_{\Geo,\bi}(f,\textbf{0},\chi):=\sum_{t\in F-\{0,1\}}\prod_{v\in\Sigma_F} \mathcal{E}_v(t),
\end{align*}
where for $v\in\Sigma_F,$ 
\begin{equation}\label{51..}
\mathcal{E}_v(t):=\int_{F_v^{\times}}\int_{F_v^{\times}}f_v\left(\begin{pmatrix}
	y_v&x_v^{-1}t\\
	x_vy_v&1
\end{pmatrix}\right)\overline{\chi}_v(y_v)d^{\times}y_vd^{\times}x_v.
\end{equation}

By Theorem 5.6 in \cite{Yan23a} (or \cite{RR05}) the orbital integrals $J^{\Reg,\RNum{2}}_{\Geo,\bi}(f,\textbf{0},\chi)$ converges absolutely. We shall establish an upper bound for it as follows. 
\begin{thmx}\label{thmD}
Let notation be as before. Then 
\begin{align*}
J^{\Reg,\RNum{2}}_{\Geo,\bi}(f,\textbf{0},\chi)\ll T^{\varepsilon}M^{\varepsilon}Q^{1+\varepsilon}\cdot \textbf{1}_{M\ll Q^2\gcd(M,Q)},
\end{align*}	
where the implied constant depends on $\varepsilon,$ $F,$ $c_v,$ and $C_v,$ $v\mid\infty$. Here $T, M,$  and $Q$ are defined in \textsection \ref{sec2.1}. In particular, $J^{\Reg,\RNum{2}}_{\Geo,\bi}(f,\textbf{0},\chi)=0$ if $M$ is large enough.
\end{thmx}
\begin{remark}
The observation that $J^{\Reg,\RNum{2}}_{\Geo,\bi}(f,\textbf{0},\chi)=0$ for large $M$ aligns with the calculation in \cite{FW09}, despite the distinct nature of the regular orbital integrals involved.
\end{remark}

\subsection{Local Estimates: unramified nonarchimedean places}
The following straightforward calculation can be found in \cite[Lemma 6.2]{Yan23c}.
\begin{lemma}\label{lem7.2}
Let $v\in \Sigma_{F,\fin}$ be such that $v\nmid\mathfrak{Q}.$ Then 
\begin{align*}
\mathcal{E}_v(t)\ll \frac{(1-e_v(1-t))(1+e_v(t)-2e_v(1-t))}{\Vol(K_v[m_v])}\textbf{1}_{\substack{e_v(t-1)\leq 0\\ e_v(t)-e_v(1-t)\geq m_v}}.
\end{align*}
Moreover, $\mathcal{E}_v(t)=1$ if $e_v(t)=e_v(1-t)=0,$ $m_v=0,$ and $v\nmid \mathfrak{D}_F.$ In particular, $\mathcal{E}_v(t)=1$ for all but finitely many $v$'s.
\end{lemma}

\subsection{Local Estimates at Ramified Places $\Sigma_{\Ram}^-$}\label{sec6.2}
In this section, we consider the case where $v\in \Sigma_{\Ram}^-,$ specifically $v \mid \mathfrak{Q}$ and $m_v < n_v$. The local integrals $\mathcal{E}_v(t)$ demonstrate unique characteristics that distinguish them from those discussed in \cite[\textsection 6.2]{Yan23c}. This distinction sets them apart from the analysis presented in the aforementioned work.
\begin{prop}\label{neg}
Let $v\in \Sigma_{\Ram}^-.$ Then 
\begin{align*}
\mathcal{E}_v(t)\ll  \begin{cases}
q_v^{m_v+k} \ \ &\text{if $e_v(1-t)=-2k$ for $m_v-n_v\leq k\leq -1$}\\
(e_v(t)-e_v(1-t)+1)q_v^{m_v} \ \ &\text{if $e_v(t)-e_v(1-t)\geq 0$}\\
(1-e_v(t))^2q_v^{m_v} \ \ &\text{if $e_v(t)\leq -1$}\\
0\ \ &\text{otherwise,}
\end{cases}
\end{align*}
where the implied constant is absolute. 
\end{prop}
\begin{proof}
By definition, $f_v\left(\begin{pmatrix}
	y_v&x_v^{-1}t\\
	x_vy_v&1
\end{pmatrix}
\right)=0$ unless 
\begin{equation}\label{48.}
\varpi_v^k\begin{pmatrix}
		1&\alpha \varpi_v^{-n_v}\\
		&1
	\end{pmatrix}\begin{pmatrix}
	y_v&x_v^{-1}t\\
	x_vy_v&1
\end{pmatrix}\begin{pmatrix}
		1&\beta \varpi_v^{-n_v}\\
		&1
	\end{pmatrix}\in K_v[m_v]
\end{equation}
for some $k\in\mathbb{Z}.$ Write $x_v=\varpi_v^{r_1}\gamma_1,$ $y_v=\varpi_v^{r_2}\gamma_2,$ where $r_1, r_2\in\mathbb{Z}$ and $\gamma_1, \gamma_2\in\mathcal{O}_v^{\times}.$ Then \eqref{48.} becomes
\begin{align*}
\varpi_v^k\begin{pmatrix}
		1&\gamma_1\alpha \varpi_v^{-n_v}\\
		&1
	\end{pmatrix}\begin{pmatrix}
	\varpi_v^{r_2}&\varpi_v^{-r_1}t\\
	\varpi_v^{r_1+r_2}&1
\end{pmatrix}\begin{pmatrix}
		1&\gamma_1\gamma_2\beta \varpi_v^{-n_v}\\
		&1
	\end{pmatrix}
\in K_v[m_v].
\end{align*}

Changing variables $\alpha\mapsto \gamma_1^{-1}\alpha,$ $\beta\mapsto \gamma_1^{-1}\gamma_2^{-1}\beta,$ the above constraint becomes 
\begin{equation}\label{49} 
\varpi_v^kY_{\alpha,\beta,r_1,r_2,t}\in K_v[m_v]
\end{equation}
for some $k\in\mathbb{Z},$ where $Y_{\alpha,\beta,r_1,r_2,t}$ is defined by 
\begin{align*}
\begin{pmatrix}
	\varpi_v^{r_2}+\alpha \varpi_v^{r_1+r_2-n_v}&(\varpi_v^{r_2}+\alpha \varpi_v^{r_1+r_2-n_v})\beta \varpi_v^{-n_v}+\varpi_v^{-r_1}t+\alpha \varpi_v^{-n_v}\\
	\varpi_v^{r_1+r_2}&1+\beta \varpi_v^{r_1+r_2-n_v}
\end{pmatrix}.
\end{align*} 

By definition the local integral $\mathcal{E}_v(t)$ becomes
\begin{align*}
\frac{1}{|\tau(\chi_v)|^2}\sum_{\alpha,\beta}\chi(\alpha)\overline{\chi}(\beta)\sum_{r_1, r_2\in\mathbb{Z}} f_v(Y_{\alpha,\beta,r_1,r_2,t};\omega_v),
\end{align*} 
where $f_v(\cdot;\omega_v)$ is defined by \eqref{5.} in \textsection\ref{2.2.2}. Note that \eqref{49} amounts to 
\begin{equation}\label{50}
\begin{cases}
2k+r_2+e_v(1-t)=0\\
k+r_1+r_2\geq m_v\\
\varpi_v^k(\varpi_v^{r_2}+\alpha \varpi_v^{r_1+r_2-n_v})\in\mathcal{O}_v\\
\varpi_v^k(1+\beta \varpi_v^{r_1+r_2-n_v})\in\mathcal{O}_v\\
\varpi_v^k\big[(\varpi_v^{r_2}+\alpha \varpi_v^{r_1+r_2-n_v})\beta \varpi_v^{-n_v}+\varpi_v^{-r_1}t+\alpha \varpi_v^{-n_v}\big]\in\mathcal{O}_v.
\end{cases}
\end{equation}

We will categorize our discussion into three cases based on the value of $k$: the case where $k\leq -1$ will be addressed in \textsection\ref{7.2.1} below, the case where $k=0$ will be addressed in \textsection\ref{7.2.2} below, and the case where $k\geq 1$ will be addressed in \textsection\ref{7.2.3} below. Proposition \ref{neg} can then be readily derived from these discussions.
\end{proof}

\subsubsection{The case that $k\leq -1$}\label{7.2.1}
Suppose $k\leq -1.$ Then \eqref{50} simplifies to  
\begin{align*}
\begin{cases}
2k+e_v(1-t)=0\\
m_v-n_v\leq k\leq -1\\
r_2=0, \ r_1=n_v\\
1+\beta\in\varpi_v^{-k}\mathcal{O}_v\\
1+\alpha\in\varpi_v^{-k}\mathcal{O}_v\\
\varpi_v^k\big[(1+\alpha )(1+\beta)\varpi_v^{-n_v}+\varpi_v^{-n_v}(t-1)\big]\in\mathcal{O}_v.
\end{cases}
\end{align*}

\begin{itemize}
	\item Suppose that $k=-n_v.$ Then $m_v=0,$ $e_v(1-t)=-2n_v,$ $\alpha=\beta=-1\pmod{\varpi_v^{n_v}}.$ So the contribution from this case is 
\begin{align*}
\frac{1}{|\tau(\chi_v)|^2}\textbf{1}_{e_v(1-t)=-2n_v}\textbf{1}_{m_v=0}=q_v^{-n_v}\textbf{1}_{e_v(1-t)=-2n_v}\textbf{1}_{m_v=0}.
\end{align*}

\item Suppose that $k>-n_v.$ Write $\alpha=-1+\varpi_v^{-k}\alpha',$ and $\beta=-1+\varpi_v^{-k}\beta',$ where $\alpha', \beta'\pmod{\varpi_v^{n_v+k}}.$ Then 
\begin{align*}
\varpi_v^k\big[(1+\alpha )(1+\beta)\varpi_v^{-n_v}+\varpi_v^{-n_v}(t-1)\big]\in\mathcal{O}_v
\end{align*}
becomes 
\begin{align*}
\alpha'\beta'+(t-1)\varpi_v^{2k}\in \varpi_v^{n_v+k}\mathcal{O}_v.
\end{align*}
So the contribution from this case is 
\begin{align*}
\frac{|\tau(\chi_v)|^{-2}}{\Vol(K_v[m_v])}\sum_{\max\{m_v-n_v,1-n_v\}\leq k\leq -1}\mathcal{S}(k),
\end{align*}
where 
\begin{align*}
\mathcal{S}(k):=\sum_{\substack{\alpha',\beta'\pmod{\varpi_v^{n_v+k}}\\
\alpha'\beta'=-(t-1)\varpi_v^{2k}\pmod{\varpi_v^{n_v+k}}}}\chi(1-\varpi_v^{-k}\alpha')\overline{\chi}(1-\varpi_v^{-k}\beta')\overline{\omega}_v(\varpi_v^{-k}\beta').
\end{align*}

Employing the trivial bound to $\mathcal{S}(k)$, we see that the corresponding contribution to $\mathcal{E}_v(t)$ in this case (i.e., $k>-n_v$) is 
\begin{align*}
\ll \Vol(K_v[m_v])^{-1}\sum_{\max\{m_v-n_v,1-n_v\}\leq k\leq -1}q_v^{k}\cdot \textbf{1}_{e_v(1-t)=-2k}.
\end{align*}
\end{itemize}

Therefore, the contribution to $\mathcal{E}_v(t)$ in the case that $k\leq -1$ is 
\begin{align*}
\ll \frac{q_v^{k}}{\Vol(K_v[m_v])}\sum_{m_v-n_v\leq k\leq -1}\textbf{1}_{e_v(1-t)=-2k}\ll q_v^{m_v+k}\sum_{m_v-n_v\leq k\leq -1}\textbf{1}_{e_v(1-t)=-2k}.
\end{align*}

\subsubsection{The case that $k=0$}\label{7.2.2}
Suppose that $k=0$ in \eqref{50}, which implies that 
\begin{equation}\label{52}
\begin{cases}
r_2+e_v(1-t)=0\\
r_1+r_2\geq m_v\\
\min\{r_2, r_1+r_2-n_v\}\geq 0\\
(\varpi_v^{r_2}+\alpha \varpi_v^{r_1+r_2-n_v})\beta \varpi_v^{-n_v}+\varpi_v^{-r_1}t+\alpha \varpi_v^{-n_v}\in\mathcal{O}_v.
\end{cases}
\end{equation}

Then $r_1+r_2\geq \max\{n_v,m_v\}=n_v,$ $e_v(t)-r_1\geq -n_v,$ and $e_v(1-t)=-r_2\leq 0.$ So $0\leq r_2=-e_v(t-1),$ and $n_v+e_v(1-t)\leq r_1\leq e_v(t)+n_v.$ Therefore, the contribution to $\mathcal{E}_v(t)$ from this case is
\begin{equation}\label{7.6}
\frac{\textbf{1}_{e_v(t)-e_v(1-t)\geq 0}}{|\tau(\chi_v)|^2\Vol(K_v[m_v])}\sum_{n_v+e_v(1-t)\leq r_1\leq e_v(t)+n_v}\mathcal{J}_1(r_1,t),
\end{equation}
where $\mathcal{J}_1(r_1,t)$ is defined by 
\begin{align*}
\sum_{\substack{\alpha,\beta\in (\mathcal{O}_v/\varpi_v^{n_v}\mathcal{O}_v)^{\times}\\ (\varpi_v^{r_2}+\alpha \varpi_v^{r_1-e_v(t-1)-n_v})\beta \varpi_v^{-n_v}+\varpi_v^{-r_1}t+\alpha \varpi_v^{-n_v}\in\mathcal{O}_v}}\chi(\alpha)\overline{\chi}(\beta)\overline{\omega}_v(1+\beta\varpi_v^{r_1-e_v(t-1)-n_v}).
\end{align*}

\begin{itemize}
\item Suppose $r_2\geq 1.$ Then $e_v(1-t)\leq -1,$ implying that $e_v(t)=e_v(1-t)=-r_2\leq -1.$ Hence, $-r_1+e_v(t)=-r_1-r_2\leq -n_v$ (from the third constraint in \eqref{52}). Along with the last condition in \eqref{52} we have $-r_1+e_v(t)\geq -n_v.$ So $-r_1+e_v(t)=-n_v,$ i.e., $r_1+r_2=n_v.$ Consequently,
\begin{align*}
\mathcal{J}_1(r_1,t)=\sum_{\substack{\alpha,\beta\in (\mathcal{O}_v/\varpi_v^{n_v}\mathcal{O}_v)^{\times}\\ (\varpi_v^{-e_v(t)}+\alpha )\beta +\varpi_v^{n_v-r_1}t+\alpha\equiv 0\pmod{\varpi_v^{n_v}}}}\chi(\alpha)\overline{\chi}(\beta)\overline{\omega}_v(1+\beta).
\end{align*}	
	
Write $t=\varpi_v^{e_v(t)}\gamma$ under the embedding $F^{\times}\hookrightarrow F_v^{\times},$ where $\gamma\in\mathcal{O}_v^{\times}.$ Then 
\begin{align*}
\mathcal{J}_1(r_1,t)=\sum_{\substack{\alpha,\beta\in (\mathcal{O}_v/\varpi_v^{n_v}\mathcal{O}_v)^{\times}\\ (\varpi_v^{-e_v(t)}+\alpha )(1+\beta)\equiv -\gamma+\varpi_v^{-e_v(t)}\pmod{\varpi_v^{n_v}}}}\chi(\alpha)\overline{\chi}(\beta)\overline{\omega}_v(1+\beta),
\end{align*}
which, after a change of variables, is equal to 
\begin{align*}
\mathcal{J}_1(r_1,t)=\sum_{\substack{\alpha,\beta\in (\mathcal{O}_v/\varpi_v^{n_v}\mathcal{O}_v)^{\times}\\ \alpha\beta\equiv -\gamma+\varpi_v^{-e_v(t)}\pmod{\varpi_v^{n_v}}}}\chi(\alpha-\varpi_v^{-e_v(t)})\overline{\chi}(\beta-1)\overline{\omega}_v(\beta).
\end{align*}

Since $-\gamma+\varpi_v^{-e_v(t)}\in \mathcal{O}_v^{\times},$ by the trivial bound,  we have $|\mathcal{J}_1(r_1,t)|\leq q_v^{n_v}.$

\item Suppose $r_2=0.$ Then $e_v(1-t)=0.$ Therefore, $\mathcal{J}_1(r_1,t)$ is equal to 
\begin{align*}
\sum_{\substack{\alpha,\beta\in (\mathcal{O}_v/\varpi_v^{n_v}\mathcal{O}_v)^{\times}\\ (1+\alpha \varpi_v^{r_1+r_2-n_v})\beta \varpi_v^{-n_v}+\varpi_v^{-r_1}t+\alpha \varpi_v^{-n_v}\in\mathcal{O}_v}}\chi(\alpha)\overline{\chi}(\beta)\overline{\omega}_v(1+\beta\varpi_v^{r_1-n_v}).
\end{align*}

Changing variable $\alpha\mapsto \overline{\alpha},$ the sum $\mathcal{J}_1(r_1,t)$ becomes
\begin{align*}
\sum_{\substack{\alpha,\beta\in (\mathcal{O}_v/\varpi_v^{n_v}\mathcal{O}_v)^{\times}\\ (\alpha+\varpi_v^{r_1-n_v})(\beta +\varpi_v^{n_v-r_1}t) \equiv t-1\pmod{\varpi_v^{n_v}}}}\overline{\chi}(\alpha)\overline{\chi}(\beta)\overline{\omega}_v(1+\beta\varpi_v^{r_1-n_v}).
\end{align*}

Changing variables $\alpha\mapsto \alpha-\varpi_v^{r_1-n_v}$ and $\beta\mapsto \beta-\varpi_v^{n_v-r_1}t,$ $\mathcal{J}_1(r_1,t)$ can be rewritten as   
\begin{align*}
\sum_{\substack{\alpha,\beta\in (\mathcal{O}_v/\varpi_v^{n_v}\mathcal{O}_v)^{\times}\\ \alpha\beta\equiv t-1\pmod{\varpi_v^{n_v}}}}\overline{\chi}(\alpha-\varpi_v^{r_1-n_v})\overline{\chi}(\beta-\varpi_v^{n_v-r_1}t)\overline{\omega}_v(1-t+\beta\varpi_v^{r_1-n_v}).
\end{align*}

Since $e_v(t-1)=0,$ then $\beta$ is uniquely determined by $\alpha.$ Hence the trivial bound yields $|\mathcal{J}_1(r_1,t)|\leq q_v^{n_v}.$ 
\end{itemize} 

Consequently, substituting the above discussions into \eqref{7.6} we then see that the contribution from this case is 
\begin{align*}
\ll \frac{\textbf{1}_{e_v(t)-e_v(1-t)\geq 0}}{|\tau(\chi_v)|^2\Vol(K_v[m_v])}\sum_{r_1}q_v^{m_v}\ll (e_v(t)-e_v(1-t)+1)q_v^{m_v}\textbf{1}_{e_v(t)-e_v(1-t)\geq m_v-n_v},
\end{align*}
where $n_v+e_v(1-t)\leq r_1\leq e_v(t)+n_v.$   

\subsubsection{The case that $k\geq 1$}\label{7.2.3}
Suppose that $k\geq 1$ in \eqref{50}, which implies that 
\begin{equation}\label{51'}
\begin{cases}
2k+r_2+e_v(1-t)=0\\
k+r_1+r_2\geq \max\{n_v, m_v\}=n_v\\
k+r_2\geq 0\\
\varpi_v^k\big[(\varpi_v^{r_2}+\alpha \varpi_v^{r_1+r_2-n_v})\beta \varpi_v^{-n_v}+\varpi_v^{-r_1}t+\alpha \varpi_v^{-n_v}\big]\in\mathcal{O}_v.
\end{cases}
\end{equation}

From the last constraint we conclude that $k-r_1+e_v(t)\geq -n_v.$ Hence
\begin{align*}
\begin{cases}
e_v(1-t)=e_v(t)\leq -k\leq -1\\
e_v(t)\leq r_2\leq -e_v(t)-2\\
n_v+e_v(t)+1\leq r_1\leq n_v\\
k\geq n_v-r_1-r_2\geq 1\\
\big[(\varpi_v^{k+r_2}+\alpha )\beta \varpi_v^{-n_v}+\varpi_v^{k-r_1}t+\alpha \varpi_v^{k-n_v}\big]\in\mathcal{O}_v.
\end{cases}
\end{align*}
 
Therefore, the contribution to $\mathcal{E}_v(t)$ from this case is
\begin{equation}\label{7.8}
\frac{\textbf{1}_{e_v(t)\leq -1}}{|\tau(\chi_v)|^2\Vol(K_v[m_v])}\sum_{r_1}\sum_{e_v(t)\leq r_2\leq -e_v(t)-2}\sum_{1\leq k\leq -e_v(t)}\mathcal{J}_2(r_1,r_2,k,t),
\end{equation}
where $n_v+e_v(t)+1\leq r_1\leq n_v,$ and $\mathcal{J}_2(r_1,r_2,k,t)$ is defined by 
\begin{align*}
\sum_{\substack{\alpha,\beta\in (\mathcal{O}_v/\varpi_v^{n_v}\mathcal{O}_v)^{\times}\\ (\varpi_v^{k+r_2}+\alpha)(\beta +\varpi_v^{k})\equiv \varpi_v^{2k+r_2}-t\varpi_v^{n_v+k-r_1}\pmod{\varpi_v^{n_v}}}}\chi(\alpha)\overline{\chi}(\beta) 
\overline{\omega}_v(1+\beta\varpi_v^{r_1+r_2-n_v}).
\end{align*}

Note that $k\geq 1,$ $\beta +\varpi_v^{k}\in(\mathcal{O}_v/\varpi_v^{n_v}\mathcal{O}_v)^{\times}.$ So $\alpha$ is uniquely determined by $\beta.$ Therefore, by the trivial bound, $|\mathcal{J}_2(r_1,r_2,k,t)|\leq q_v^{n_v}.$ Along with \eqref{7.8}, the contribution to $\mathcal{E}_v(t)$ from this case is 
\begin{align*}
\sum_{n_v+e_v(t)+1\leq r_1\leq n_v}\sum_{e_v(t)\leq r_2\leq -e_v(t)-2}q_v^{m_v}\textbf{1}_{e_v(t)\leq -1}\ll (1-e_v(t))^2q_v^{m_v}\textbf{1}_{e_v(t)\leq -1}.
\end{align*}

\begin{remark}
A more refined bound can be derived in the case where $k\geq 0$ by estimating the character sums nontrivially. However, it becomes apparent that the contribution from the $k\geq 0$ case is overshadowed by the contribution from $k\leq -1$. Therefore, there is no necessity to further reduce the error term. 
\end{remark}

\subsection{Local Estimates at Ramified Places $\Sigma_{\Ram}^+$} 
Consider $v\in \Sigma_{\Ram}^+$, which means $v\mid\mathfrak{Q}$ and $m_v\geq n_v$, where $m_v=e_v(\mathfrak{M})$ and $n_v=r_{\chi_v}$ (cf. \textsection\ref{2.1.5}). 
\begin{prop}\label{4}
Let $v\in \Sigma_{\Ram}^+.$ Then 
\begin{align*}
\mathcal{E}_v(t)\ll 
\begin{cases}
(1-e_v(t))^2q_v^{m_v} \ \ &\text{if $e_v(t)\leq -1,$ $m_v=n_v,$}\\
(e_v(t)-e_v(1-t)+1+m_v-n_v)q_v^{m_v}\ \ &\text{if $e_v(t)\geq m_v-n_v,$}\\
0\ \ &\text{otherwise,}
\end{cases}
\end{align*}
where the implied constant depends at most on $F_v$. 
\end{prop}
\begin{proof}
Consider the notation used in the proof of Proposition \ref{neg} in \textsection\ref{sec6.2} (or in \cite[\textsection 6.2]{Yan23c}). Since $m_v\geq n_v\geq 1,$ the  constraints \eqref{50} can be simplified as follows:
\begin{equation}\label{7.9}
\begin{cases}
2k+r_2+e_v(1-t)=0\\
k+r_1+r_2\geq m_v\\
\varpi_v^k(\varpi_v^{r_2}+\alpha \varpi_v^{r_1+r_2-n_v})\in\mathcal{O}_v^{\times}\\
\varpi_v^k(1+\beta \varpi_v^{r_1+r_2-n_v})\in\mathcal{O}_v^{\times}\\
\varpi_v^k\big[(\varpi_v^{r_2}+\alpha \varpi_v^{r_1+r_2-n_v})\beta \varpi_v^{-n_v}+\varpi_v^{-r_1}t+\alpha \varpi_v^{-n_v}\big]\in\mathcal{O}_v.
\end{cases}
\end{equation}
By considering the second and fourth constraints in \eqref{7.9}, we deduce that $k\geq 0$. We can now proceed to examine the following two cases.

\subsubsection{The case that $k=0$}
Suppose that $k=0$ in \eqref{7.9}, which implies that 
\begin{equation}\label{52..}
\begin{cases}
r_2+e_v(1-t)=0\\
r_1+r_2\geq m_v\\
\min\{r_2, r_1+r_2-n_v\}=0\\
(\varpi_v^{r_2}+\alpha \varpi_v^{r_1+r_2-n_v})\beta \varpi_v^{-n_v}+\varpi_v^{-r_1}t+\alpha \varpi_v^{-n_v}\in\mathcal{O}_v.
\end{cases}
\end{equation}

Then $r_1+r_2\geq \max\{n_v,m_v\}=m_v,$ $e_v(t)-r_1\geq -n_v,$ and $e_v(1-t)=-r_2\leq 0.$ So $0\leq r_2=-e_v(t-1),$ and $m_v+e_v(1-t)\leq r_1\leq e_v(t)+n_v.$ Therefore, the contribution to $\mathcal{E}_v(t)$ from this case is
\begin{equation}\label{7.6}
\frac{\textbf{1}_{e_v(t)-e_v(1-t)\geq m_v-n_v}}{|\tau(\chi_v)|^2\Vol(K_v[m_v])}\sum_{m_v+e_v(1-t)\leq r_1\leq e_v(t)+n_v}\mathcal{J}_1(r_1,t),
\end{equation}
where $\mathcal{J}_1(r_1,t)$ is defined by 
\begin{align*}
\sum_{\substack{\alpha,\beta\in (\mathcal{O}_v/\varpi_v^{n_v}\mathcal{O}_v)^{\times}\\ (\varpi_v^{r_2}+\alpha \varpi_v^{r_1-e_v(t-1)-n_v})\beta \varpi_v^{-n_v}+\varpi_v^{-r_1}t+\alpha \varpi_v^{-n_v}\in\mathcal{O}_v}}\chi(\alpha)\overline{\chi}(\beta)\overline{\omega}_v(1+\beta\varpi_v^{r_1-e_v(t-1)-n_v}).
\end{align*}	

By the trivial bound (as in \textsection\ref{7.2.2}) the sum in \eqref{7.6} is 
\begin{align*}
\ll (e_v(t)-e_v(1-t)+1+m_v-n_v)q_v^{m_v}\textbf{1}_{e_v(t)-e_v(1-t)\geq m_v-n_v}.
\end{align*}

\subsubsection{The case that $k\geq 1$} Suppose that $k\geq 1$ in \eqref{50}, which implies that 
\begin{equation}\label{51'}
\begin{cases}
2k+r_2+e_v(1-t)=0\\
k+r_1+r_2\geq n_v\\
k+r_1+r_2-m_v=0\\
k+r_2\geq 0\\
\varpi_v^k\big[(\varpi_v^{r_2}+\alpha \varpi_v^{r_1+r_2-n_v})\beta \varpi_v^{-n_v}+\varpi_v^{-r_1}t+\alpha \varpi_v^{-n_v}\big]\in\mathcal{O}_v.
\end{cases}
\end{equation}

Since $m_v\geq n_v,$ then by the second and the third constraints in \eqref{51'} we have $m_v=n_v.$ From the last constraint we conclude that $k-r_1+e_v(t)\geq -n_v.$ Hence
\begin{align*}
\begin{cases}
e_v(1-t)=e_v(t)\leq -k\leq -1,\ \ m_v=n_v\\
e_v(t)\leq r_2\leq -e_v(t)-2\\
n_v+e_v(t)+1\leq r_1\leq n_v\\
k= n_v-r_1-r_2\geq 1\\
\big[(\varpi_v^{k+r_2}+\alpha )\beta \varpi_v^{-n_v}+\varpi_v^{k-r_1}t+\alpha \varpi_v^{k-n_v}\big]\in\mathcal{O}_v.
\end{cases}
\end{align*}
 
As in \textsection\ref{7.2.3},  the contribution to $\mathcal{E}_v(t)$ from this case is
\begin{align*}
\frac{\textbf{1}_{e_v(t)\leq -1}\cdot \textbf{1}_{m_v=n_v}}{|\tau(\chi_v)|^2\Vol(K_v[m_v])}\sum_{n_v+e_v(t)+1\leq r_1\leq n_v}\sum_{e_v(t)\leq r_2\leq -e_v(t)-2}\mathcal{J}_2(r_1,r_2,k,t),
\end{align*}
where $\mathcal{J}_2(r_1,r_2,k,t)$ is defined by 
\begin{align*}
\sum_{\substack{\alpha,\beta\in (\mathcal{O}_v/\varpi_v^{n_v}\mathcal{O}_v)^{\times}\\ (\varpi_v^{k+r_2}+\alpha)(\beta +\varpi_v^{k})\equiv \varpi_v^{2k+r_2}-t\varpi_v^{n_v+k-r_1}\pmod{\varpi_v^{n_v}}}}\chi(\alpha)\overline{\chi}(\beta) 
\overline{\omega}_v(1+\beta\varpi_v^{r_1+r_2-n_v}).
\end{align*}

By trivial bound the contribution to $\mathcal{E}_v(t)$ in this case is 
\begin{align*}
\ll (1-e_v(t))^2q_v^{m_v}\textbf{1}_{e_v(t)\leq -1}\textbf{1}_{m_v=n_v}.
\end{align*}

Therefore, Proposition \ref{4} follows. 
\end{proof}

\subsection{Local Estimates: archimedean}\label{sec6.4}
Let $v\mid \infty.$ Define by 
\begin{equation}\label{6.21}
\mathcal{E}_v^{\dagger}:=\int_{F_v^{\times}}\int_{F_v}\max_{t\in F-\{0,1\}}\Big|f_v\left(\begin{pmatrix}
	y_v&x_v^{-1}t\\
	x_vy_v&1
\end{pmatrix}\right)\Big|dx_vd^{\times}y_v.
\end{equation}

By \cite[Lemma 6.8]{Yan23c} we have the following estimate. 
\begin{lemma}\label{lem6.8}
Let notation be as before. Let $v\mid \infty.$ Then 
\begin{equation}\label{6.22}
\mathcal{E}_v^{\dagger}\ll T_v^{\varepsilon},
\end{equation}
where the implied constant depends on $\varepsilon,$ $F,$ $c_v,$ and $C_v$ defined in \textsection\ref{2.1.2}.  
\end{lemma}

\subsection{Bounding Regular Orbital Integrals: Proof of Theorem \ref{thmD}} 
\subsubsection{The support of the rationals $t\in F-\{0,1\}$}
\begin{lemma}\label{lem10}
Let notation be as before. Suppose $t\in F-\{0,1\}.$ Let $f$ be the test function defined in \textsection\ref{3.2.5}. Let 
\begin{equation}\label{13}
\mathfrak{X}(\mathfrak{Q},f):=\bigg\{\xi\in F^{\times}\cap \prod_{v\in\Sigma_{\Ram}^-}\mathfrak{p}_v^{-2(n_v-m_v)}\prod_{v\nmid\mathfrak{Q}}\mathfrak{p}_v^{m_v}\mathcal{O}_F:\ |\xi|_{v}\ll 1,\ v\mid \infty\bigg\},
\end{equation}
where the implied constant depends only on $\supp f_{\infty}.$ Then the integral  $\prod_{v\in\Sigma_F} \mathcal{E}_v(t)$ converges absolutely and it vanishes unless $\frac{t}{t-1}\in \mathfrak{X}(\mathfrak{Q},f).$
\end{lemma}
\begin{proof}
Recall the definition \eqref{51..}: for $v\in\Sigma_F,$
\begin{align*}
\mathcal{E}_v(t):=\int_{F_v^{\times}}\int_{F_v^{\times}}f_v\left(\begin{pmatrix}
	y_v&x_v^{-1}t\\
	x_vy_v&1
\end{pmatrix}\right)\overline{\chi}_v(y_v)d^{\times}y_vd^{\times}x_v.
\end{align*}

By Lemma \ref{lem7.2} the integral $\mathcal{E}_v(t)=1$ for all but finitely many $v$'s. It then follows from Propositions \ref{neg} and \ref{4}, and Lemma \ref{lem6.8} that $\prod_{v\in\Sigma_F} \mathcal{E}_v(t)$ converges absolutely and it is vanishing unless
\begin{equation}\label{6.27}
\begin{cases}
e_v(t)-e_v(t-1)\geq m_v,\ & \text{if $v\nmid\mathfrak{Q},$}\\
e_v(t)-e_v(t-1)\geq 0,\ & \text{if $v\in\Sigma_{\Ram}^+.$}\\
e_v(t)-e_v(t-1)\geq -2(n_v-m_v),\ & \text{if $v\in\Sigma_{\Ram}^-.$}
\end{cases}
\end{equation}

Since $t/(t-1)\in F-\{0,1\},$ then \eqref{13} follows from \eqref{6.27}.
\end{proof}

\subsubsection{Estimate of nonarchimedean integrals}\label{6.5.2}
Fix an ideal $\mathfrak{R}\subset \mathcal{O}_{F}$ with the property that $e_v(\mathfrak{R})=m_v$ for $v\nmid\mathfrak{Q},$ and   $e_v(\mathfrak{R})=0$ for all $v<\infty$ and $v\mid\mathfrak{Q}.$ 

Fix an ideal $\mathfrak{N}\subset \mathcal{O}_{F}$ with the property that $e_v(\mathfrak{N})=n_v-m_v$ for $v\in\Sigma_{\Ram}^-,$ and   $e_v(\mathfrak{R})=0$ for all $v<\infty$ and $v\not\in\Sigma_{\Ram}^-.$

For $t\in F-\{0,1\}$ with $t/(t-1)\in\mathfrak{X}(\mathfrak{Q},f)$ (cf. \eqref{13}), we may write 
\begin{equation}\label{6.28}
t/(t-1)=u,\ \ u\in \mathfrak{R}\mathfrak{N}^{-2}\mathcal{O}_F.
\end{equation}
Then $1/(t-1)=u-1.$

 \begin{lemma}\label{lem6.10}
 Let notation be as above. Let $\mathcal{E}_v(t)$ be defined by \eqref{51..}. Set $\mathcal{E}_{\fin}(t):=\prod_{v<\infty}|\mathcal{E}_v(t)|.$ Let $t/(t-1)=u\in \mathfrak{R}\mathfrak{N}^{-2}\mathcal{O}_F$ be as in \eqref{6.28}. Then $\mathcal{E}_{\fin}(t)$ is 
\begin{equation}\label{7.21}
\ll (MQN_F(u(u-1)))^{\varepsilon}M\prod_{\substack{v\in\Sigma_{F,\fin}\\ v\nmid\mathfrak{Q}}}\textbf{1}_{e_v(u)\geq m_v}\prod_{v\in\Sigma_{\Ram}^-}\mathcal{J}_v^-(u)\prod_{v\in\Sigma_{\Ram}^+}\mathcal{J}_v^+(u),
\end{equation}
where $M=N_F(\mathfrak{M})$ (cf. \eqref{2.1.5}), and 
\begin{align*}
\mathcal{J}_v^-(u):=& \textbf{1}_{e_v(u)\geq 0}+\sum_{m_v-n_v\leq k\leq -1}q_v^{k}\textbf{1}_{e_v(u-1)=2k},\\
\mathcal{J}_v^+(u):=&\textbf{1}_{\substack{e_v(u-1)\geq 1}}\textbf{1}_{m_v=n_v}+\textbf{1}_{\substack{e_v(u)\geq m_v-n_v}}.
\end{align*}
Here the implied constant in \eqref{7.21} depends on $F$ and $\varepsilon.$
\end{lemma}
\begin{proof}
By Lemma \ref{lem7.2} we have $\mathcal{E}_v(t)=1$ if $e_v(t)=e_v(1-t)=0,$ $m_v=0,$ $n_v=0,$ and $v\nmid \mathfrak{D}_F.$ There are finitely many  remaining places 
$$
v\in\mathcal{V}:=\{\text{$v\in \Sigma_{F,\fin}:$ $v\mid \mathfrak{M}\mathfrak{N}$ or $e_v(t)\neq 0$ or $e_v(t-1)\neq 0$}\}.
$$

Let us denote the expression $\boldsymbol{\alpha}$ as follows:
\begin{align*}
\prod_{v\in \mathcal{V}\cap \Sigma_{\Ram}^+}n_v^2(|e_v(t)-e_v(t-1)|+1)^2\prod_{v\in \mathcal{V}-\Sigma_{\Ram}^+}(1+|e_v(t)|+2|e_v(t-1)|)^2,
\end{align*} 
where the terms in the product dominate coefficients in Lemma \ref{lem7.2}, Propositions \ref{neg} and \ref{4}. Using \eqref{6.28}, we observe that $e_v(u)\geq -e_v(\mathfrak{Q})$. Consequently, we have the estimate: 
\begin{align*}
\boldsymbol{\alpha}\ll (MQ)^{2\varepsilon}\cdot (N_F(u)N_F(u-1))^{\varepsilon},
\end{align*} 
where the implied constants depends on $\varepsilon.$ As a consequence, \eqref{7.21} follows from Lemma \ref{lem7.2}, Propositions \ref{neg} and \ref{4}.
\end{proof}

For $x_{\infty}=\otimes_{v\mid\infty}x_v\in F_{\infty}.$ For $t\in \mathfrak{X}(\mathfrak{Q},f),$ parametrize $t/(t-1)$ via \eqref{6.28}. Let 
\begin{equation}\label{6.33}
\mathcal{C}(x_{\infty}):=\sum_{\substack{t\in F-\{0,1\},\ \frac{t}{t-1}=u\in \mathfrak{X}(\mathfrak{Q},f)\\
|\frac{t}{t-1}|_v\ll |x_v|_v,\ v\mid\infty}}\mathcal{E}_{\fin}(t).
\end{equation}

\begin{lemma}\label{lem6.11}
Let notation be as before. Let $x_{\infty}\in F_{\infty}^{\times}.$ Let $\mathcal{C}(x_{\infty})$ be defined by \eqref{6.33}. Then
 \begin{equation}\label{6.34}
\mathcal{C}(x_{\infty})\ll_{\varepsilon,F}(MQ(1+|x_{\infty}|_{\infty}))^{\varepsilon}\cdot |x_{\infty}|_{\infty}\cdot Q\cdot \textbf{1}_{M\ll Q^2\gcd(M,Q) |x_{\infty}|_{\infty}},
\end{equation}
where the implied constant depends on $\varepsilon$ and $F.$
\end{lemma}
\begin{proof}
Note that $u\in \mathfrak{R}\mathfrak{N}^{-2}\mathcal{O}_F-\{0,1\}$ and $N_F(u)\ll |x_{\infty}|_{\infty}.$ Hence, $N_F(\mathfrak{R}\mathfrak{N}^{-2})\ll |x_{\infty}|_{\infty},$ i.e., $M/\gcd(M,Q)\ll Q^2|x_{\infty}|_{\infty}.$  By Lemma \ref{lem6.10}, we have
\begin{equation}\label{7.211}
\mathcal{C}(x_{\infty})\ll (MQ(1+|x_{\infty}|_{\infty}))^{\varepsilon} \mathcal{S}(x_{\infty})\cdot \prod_{v\in \Sigma_{F,\fin}}q_v^{m_v}.
\end{equation}
where the auxiliary sum  $\mathcal{S}(x_{\infty})$ is defined by 
\begin{align*}
\mathcal{S}(x_{\infty}):=\sum_{\substack{u\in \mathfrak{R}\mathfrak{N}^{-2}\mathcal{O}_F\cap  F^{\times}\\
|u|_v\ll |x_v|_v,\ v\mid\infty\\ e_v(u)\geq m_v,\ v<\infty,\ v\nmid \mathfrak{Q}}}\mathcal{S}^+(u)\mathcal{S}^-(u).
\end{align*}
Here the integral ideals $\mathfrak{R}$ and $\mathfrak{N}$ are defined in \textsection\ref{6.5.2}, and 
\begin{align*}
\mathcal{S}^+(u):=&\prod_{v\in\Sigma_{\Ram}^+}\Bigg[1_{e_v(u-1)\geq 1}\cdot \textbf{1}_{m_v=n_v}+\textbf{1}_{\substack{e_v(u)\geq m_v-n_v}}\Bigg],\\
\mathcal{S}^-(u):=&\prod_{v\in\Sigma_{\Ram}^-}\Bigg[\textbf{1}_{e_v(u)\geq 0}+\sum_{m_v-n_v\leq k\leq -1}q_v^{k}\textbf{1}_{e_v(u-1)=2k}\Bigg].
\end{align*}

We proceed to deal with $\mathcal{S}(x_{\infty}).$ Let $\mathfrak{Q}^+:=\prod_{v\in\Sigma_{\Ram}^+}\mathfrak{p}_v$ and $\mathfrak{Q}^-:=\prod_{v\in\Sigma_{\Ram}^-}\mathfrak{p}_v.$ Expanding the products $\mathcal{S}^+(u)\mathcal{S}^-(u)$ we obtain 
\begin{align*}
\mathcal{S}(x_{\infty})=\sum_{\substack{\mathfrak{a}_1\mathfrak{a}_2=\mathfrak{Q}^+\\ m_v=n_v,\ \forall \ v\mid\mathfrak{a}_1}}\sum_{\substack{\mathfrak{b}_3\mathfrak{b}_4=\mathfrak{Q}^-}}\mathcal{S}(\mathfrak{a}_1,\mathfrak{b}_3),
\end{align*}
where 
\begin{align*}
\mathcal{S}(\mathfrak{a}_1,\mathfrak{b}_3):=\sum_{\substack{u\in \mathfrak{R}\mathfrak{N}^{-2}\mathcal{O}_F\cap  F^{\times}\\
|u|_v\ll |x_v|_v,\ v\mid\infty \\
e_v(u)\geq m_v,\ v<\infty,\ v\nmid \mathfrak{Q}\\
e_{v_1}(u-1)\geq 1,\ v_1\mid\mathfrak{a}_1\\
e_{v_2}(u)\geq m_{v_2}-n_{v_2},\ v_2\mid\mathfrak{a}_2\\
e_{v_3}(u)\geq 0,\ v_3\mid\mathfrak{b}_3
}}\prod_{{v_4}\mid\mathfrak{b}_4}\sum_{m_{v_4}-n_{v_4}\leq k\leq -1}q_{v_4}^{k}\textbf{1}_{e_{v_4}(u-1)=2k}.
\end{align*}

Write $\mathfrak{b}_4=\mathfrak{Q}^-\mathfrak{b}_3^{-1}=\prod_{v\in\mathcal{V}_4}\mathfrak{p}_v,$ where $\mathcal{V}_4=\{v_1',\cdots,v_l'\}$ is a subset of $\Sigma_{\Ram}^-.$ Denote by $\underline{k}=(k_1,\cdots,k_l)\in \mathbb{Z}^{l}.$ Then 
\begin{align*}
\prod_{{v_4}\mid\mathfrak{b}_4}\sum_{m_{v_4}-n_{v_4}\leq k\leq -1}q_{v_4}^{k}\textbf{1}_{e_{v_4}(u-1)=2k}=\sum_{\substack{\underline{k}=(k_1,\cdots,k_l)\\
m_{v_i'}-n_{v_i'}\leq k_i\leq -1,\ 1\leq i\leq l}}\prod_{j=1}^lq_{v_j'}^{k_j}\textbf{1}_{\substack{e_{v_j'}(u-1)=2k_j}}.
\end{align*}  

Therefore, 
\begin{align*}
\mathcal{S}(x_{\infty})=\sum_{\substack{\mathfrak{a}_1\mathfrak{a}_2=\mathfrak{Q}^+\\ m_v=n_v,\ \forall \ v\mid\mathfrak{a}_1}}\sum_{\substack{\mathfrak{b}_3\mathfrak{b}_4=\mathfrak{Q}^-}}\sum_{\substack{\underline{k}=(k_1,\cdots,k_l)\\
m_{v_i'}-n_{v_i'}\leq k_i\leq -1,\ 1\leq i\leq l}}\prod_{j=1}^lq_{v_j'}^{k_j}\cdot \mathcal{S}^{\dagger}(x_{\infty}),
\end{align*}
where 
\begin{align*}
\mathcal{S}^{\dagger}(x_{\infty}):=\sum_{\substack{u\in \mathfrak{R}\mathfrak{N}^{-2}\mathcal{O}_F\cap  F^{\times}\\
|u|_v\ll |x_v|_v,\ v\mid\infty \\
e_v(u)\geq m_v,\ v<\infty,\ v\nmid \mathfrak{Q}\\
e_{v_1}(u-1)\geq 1,\ v_1\mid\mathfrak{a}_1\\
e_{v_2}(u)\geq m_{v_2}-n_{v_2},\ v_2\mid\mathfrak{a}_2\\
e_{v_3}(u)\geq 0,\ v_3\mid\mathfrak{b}_3\\
e_{v_j'}(u-1)=2k_j,\ 1\leq j\leq l
}}1.
\end{align*}

By counting rational lattice points in a bounded region, we have 
\begin{align*}
\mathcal{S}^{\dagger}(x_{\infty})\ll \sum_{\substack{u\in \mathfrak{R}\mathfrak{N}^{-2}\mathcal{O}_F\cap  F^{\times}\\
|u|_v\ll |x_v|_v,\ v\mid\infty \\
e_v(u)\geq m_v,\ v<\infty,\ v\nmid \mathfrak{Q}\\
e_{v_2}(u)\geq m_{v_2}-n_{v_2},\ v_2\mid\mathfrak{a}_2\\
e_{v_3}(u)\geq 0,\ v_3\mid\mathfrak{b}_3\\
e_{v_j'}(u)=2k_j
}}1\ll |x_{\infty}|_{\infty}\prod_{\substack{v\in\Sigma_{F,\fin}\\ v\nmid \mathfrak{Q}}}q_v^{-m_v}\prod_{v_2\mid \mathfrak{a}_2}q_{v_2}^{n_{v_2}-m_{v_2}}\prod_{j=1}^lq_{v_j'}^{-2k_j}.
\end{align*}

Therefore, $\mathcal{S}(x_{\infty})$ is majorized by 
\begin{align*}
|x_{\infty}|_{\infty}\prod_{\substack{v\in\Sigma_{F,\fin}\\ v\nmid \mathfrak{Q}}}\frac{1}{q_v^{m_v}}\sum_{\substack{\mathfrak{a}_1\mathfrak{a}_2=\mathfrak{Q}^+\\ m_v=n_v,\ \forall \ v\mid\mathfrak{a}_1}}\prod_{v_2\mid \mathfrak{a}_2}q_{v_2}^{n_{v_2}-m_{v_2}}\sum_{\substack{\mathfrak{b}_3\mathfrak{b}_4=\mathfrak{Q}^-}}\sum_{\substack{\underline{k}=(k_1,\cdots,k_l)\\
m_{v_i'}-n_{v_i'}\leq k_i\leq -1,\ 1\leq i\leq l}}\prod_{j=1}^l\frac{1}{q_{v_j'}^{k_j}}.
\end{align*}

Notice that 
\begin{align*}
\sum_{\substack{\mathfrak{a}_1\mathfrak{a}_2=\mathfrak{Q}^+\\ m_v=n_v,\ \forall \ v\mid\mathfrak{a}_1}}\prod_{v_2\mid \mathfrak{a}_2}q_{v_2}^{n_{v_2}-m_{v_2}}=\sum_{\substack{v\mid\mathfrak{Q}^+}}q_{v_2}^{n_{v_2}-m_{v_2}}\sum_{\substack{\mathfrak{a}_1\mathfrak{a}_2=\mathfrak{Q}^+\\ m_v=n_v,\ \forall \ v\mid\mathfrak{a}_1}}1\ll Q^{\varepsilon}\sum_{\substack{v\mid\mathfrak{Q}^+}}q_{v_2}^{n_{v_2}-m_{v_2}},
\end{align*}
and 
\begin{align*}
\sum_{\substack{\mathfrak{b}_3\mathfrak{b}_4=\mathfrak{Q}^-}}\sum_{\substack{\underline{k}=(k_1,\cdots,k_l)\\
m_{v_i'}-n_{v_i'}\leq k_i\leq -1,\ 1\leq i\leq l}}\prod_{j=1}^lq_{v_j'}^{-k_j}\ll \prod_{v\mid\mathfrak{Q}^-}q_v^{n_v-m_v}\sum_{\substack{\mathfrak{b}_3\mathfrak{b}_4=\mathfrak{Q}^-}}\sum_{\substack{\underline{k}=(k_1,\cdots,k_l)\\
m_{v_i'}-n_{v_i'}\leq k_i\leq -1,\ 1\leq i\leq l}}1,
\end{align*}
which is $\ll Q^{\varepsilon}\prod_{v\mid\mathfrak{Q}^-}q_v^{n_v-m_v}.$ Therefore, 
\begin{equation}\label{7.22}
\mathcal{S}(x_{\infty})\ll |x_{\infty}|_{\infty}Q^{\varepsilon}\prod_{\substack{v\in\Sigma_{F,\fin}\\ v\nmid \mathfrak{Q}}}q_v^{-m_v}\prod_{\substack{v\mid\mathfrak{Q}}}q_{v}^{n_{v}-m_{v}}.
\end{equation}

Then \eqref{6.34} follows from substituting \eqref{7.22} into \eqref{7.211}.
\end{proof}

\subsubsection{Proof of Theorem \ref{thmD}} 
Recall the definition \eqref{17...} in \textsection\ref{sec3.7}: 
\begin{align*}
J^{\Reg,\RNum{2}}_{\Geo,\bi}(f,\textbf{0},\chi)=\sum_{t\in F-\{0,1\}}\int_{\mathbb{A}_F^{\times}}\int_{\mathbb{A}_F^{\times}}f\left(\begin{pmatrix}
	y&x^{-1}t\\
	xy&1
\end{pmatrix}\right)\overline{\chi}(y)d^{\times}yd^{\times}x.
\end{align*}

So the regular orbital integrals $J^{\Reg,\RNum{2}}_{\Geo,\bi}(f,\textbf{0},\chi)$ is 
\begin{align*}
\ll \int_{F_{\infty}^{\times}}\int_{F_{\infty}^{\times}}\sum_{\substack{t\in F-\{0,1\}\\ \frac{t}{t-1}\in \mathfrak{X}(\mathfrak{Q},f)}}\mathcal{E}_{\fin}(t)\Big|f_{\infty}\left(\begin{pmatrix}
	y_{\infty}&x_{\infty}^{-1}t\\
	x_{\infty}y_{\infty}&1
\end{pmatrix}\right)\Big|d^{\times}y_{\infty}d^{\times}x_{\infty},
\end{align*}
where $\mathcal{E}_{\fin}(t):=\prod_{v<\infty}|\mathcal{E}_v(t)|.$

By the support of $f_{\infty}$ (cf. \eqref{245} in \textsection\ref{3.2.1}), we have 
\begin{equation}\label{6.32}
f_{\infty}\left(\begin{pmatrix}
	y_{\infty}&x_{\infty}^{-1}t\\
	x_{\infty}y_{\infty}&1
\end{pmatrix}\right)=0
\end{equation}
unless $y_{\infty}\asymp 1,$ $|x_{v}|_v\ll 1,$ and $\big|\frac{t}{t-1}\big|_{v}\ll |x_v|_{v},$ for all $v\mid\infty.$ Write 
$$t/(t-1)=u\mathfrak{N}^{-2}\mathfrak{R}$$ with $u\in\mathcal{O}_F$ as in \eqref{6.28}. Then 
 $J^{\Reg,\RNum{2}}_{\Geo,\bi}(f,\textbf{0},\chi)$ is 
\begin{align*}
\ll \int_{F_{\infty}^{\times}}\int_{1+o(1)}\textbf{1}_{\substack{|x_v|_v\ll 1\\ v\mid\infty}}\cdot \mathcal{C}(x_{\infty})\cdot \max_{t\in \mathfrak{X}(\mathfrak{Q},f)}\Big|f_{\infty}\left(\begin{pmatrix}
	y_{\infty}&x_{\infty}^{-1}t\\
	x_{\infty}y_{\infty}&1
\end{pmatrix}\right)\Big|d^{\times}y_{\infty}d^{\times}x_{\infty},
\end{align*}
where $\mathcal{C}(x_{\infty})$ is defined by \eqref{6.33}. Note that $|x_v|_v\ll 1$ for $v\mid\infty,$ yielding that $|x_{\infty}|_{\infty}\ll 1.$ Hence, we may replace $\textbf{1}_{M\ll Q^2\gcd(M,Q) |x_{\infty}|_{\infty}}$ with $\textbf{1}_{M\ll Q^2\gcd(M,Q)}$ in Lemma \ref{lem6.11}. As a consequence, we have 
\begin{align*}
J^{\Reg,\RNum{2}}_{\Geo,\bi}(f,\textbf{s}_0,\chi)\ll_{\varepsilon}(MQ)^{\varepsilon}\cdot Q\cdot \textbf{1}_{M\ll Q^2\gcd(M,Q)}\cdot \prod_{v\mid\infty}\mathcal{E}_v^{\dagger},
\end{align*}
where $\mathcal{E}_v^{\dagger}$ is defined by \eqref{6.21}. By Lemma \ref{lem6.8}, the above bound becomes
\begin{align*}
J^{\Reg,\RNum{2}}_{\Geo,\bi}(f,\textbf{s}_0,\chi)\ll T^{\varepsilon}M^{\varepsilon}Q^{1+\varepsilon}\cdot \textbf{1}_{M\ll Q^2\gcd(M,Q)},
\end{align*}	
where the implied constant depends on $\varepsilon,$ $F,$ $c_v,$ and $C_v,$ $v\mid\infty$.

\section{Proof of Main Results}\label{sec7} 
Recall the intrinsic data in \textsection\ref{sec2.1}. Let $F$ be a number field. Let $\chi=\otimes_v\chi_v$ be a primitive unitary Hecke character of $F^{\times}\backslash\mathbb{A}_F^{\times}$. 
\subsection{The Spectral Side}\label{sec7.1}
Recall the lower bound of $J_{\Spec}^{\Reg,\heartsuit}(f,\textbf{0},\chi)$ in \textsection\ref{sec3}.
 \thmf*

\subsection{The Geometric Side}\label{sec7.2}
Recall the geometric side \eqref{geom} in \textsection\ref{rref}: 
\begin{align*}
J_{\Geo}^{\Reg,\heartsuit}(f,\textbf{0},\chi)=J^{\Reg}_{\Geo,\sm}(f,\chi)+J^{\Reg}_{\Geo,\du}(f,\chi)+J^{\Reg,\RNum{2}}_{\Geo,\bi}(f,\textbf{0},\chi).
\end{align*}
 
\begin{prop}\label{prop7.2}
Let notation be as before. Then 
\begin{equation}\label{8.1}
J_{\Geo}^{\Reg,\heartsuit}(f,\textbf{0},\chi)\ll T^{\frac{1}{2}+\varepsilon}M^{1+\varepsilon}+T^{\varepsilon}M^{\varepsilon}Q^{1+\varepsilon}\cdot \textbf{1}_{M\ll Q^2\gcd(M,Q)},
\end{equation}
where the implied constant depends on $\varepsilon,$ $F,$ $c_v,$ and $C_v,$ $v\mid\infty$ (cf. \textsection\ref{2.1.2}).
\end{prop}
\begin{proof}
By Propositions \ref{prop12} and \ref{prop17}, we have 
\begin{align*}
J^{\Reg}_{\Geo,\sm}(f,\chi)+J^{\Reg}_{\Geo,\du}(f,\chi)\ll_{\varepsilon} M^{1+\varepsilon}T^{\frac{1}{2}+\varepsilon},
\end{align*}
where the implied constant depends only on $F,$ $\varepsilon,$ and $c_v,$ $C_v$ at $v\mid\infty,$ cf. \textsection \ref{2.1.2}. Moreover, by Theorem \ref{thmD} we have 
\begin{align*}
J^{\Reg,\RNum{2}}_{\Geo,\bi}(f,\textbf{0},\chi)\ll T^{\varepsilon}M^{\varepsilon}Q^{1+\varepsilon}\cdot \textbf{1}_{M\ll Q^2\gcd(M,Q)}.
\end{align*}	

The estimate \eqref{8.1} follows from the above inequalities. 
\end{proof}

\subsection{Put It All Together: Proof of Main Results}\label{sec7.3}

Substituting Theorem \ref{thm6} and Proposition \ref{prop7.2} into the regularized relative trace formula $J_{\Spec}^{\Reg,\heartsuit}(f,\textbf{0},\chi)=J_{\Geo}^{\Reg,\heartsuit}(f,\textbf{0},\chi)$ (cf. Corollary \ref{cor3.3} in \textsection\ref{rref}), we obtain the following.
\begin{thmx}\label{thmE}
Let the notation be as before. Denote by $\mathcal{A}_0(\Pi_{\infty},\mathfrak{M};\chi_{\infty},\omega)$ the set of cuspidal representations and $\mathcal{X}_0(\Pi_{\infty},\mathfrak{M};\chi_{\infty},\omega)$ the set of Hecke characters, as defined in \textsection\ref{auto}. Then 
\begin{equation}\label{8.2}
\sum_{\pi}|L(1/2,\pi\times\chi)|^2\ll T^{1+\varepsilon}M^{1+\varepsilon}Q^{\varepsilon}+T^{\frac{1}{2}+\varepsilon}M^{\varepsilon}Q^{1+\varepsilon}\cdot \textbf{1}_{M\ll Q^2\gcd(M,Q)},
\end{equation}
where $\pi \in\mathcal{A}_0(\Pi_{\infty},\mathfrak{M};\chi_{\infty},\omega),$ and 
\begin{align*}
\sum_{\eta}\int_{\mathbb{R}}\frac{|L(1/2+it,\eta\chi)L(1/2+it,\omega\overline{\eta}\chi)|^2}{|L(1+2it,\omega\overline{\eta}^2)|^2}dt&\ll T^{1+\varepsilon}M^{1+\varepsilon}Q^{\varepsilon}\\
&+T^{\frac{1}{2}+\varepsilon}M^{\varepsilon}Q^{1+\varepsilon}\cdot \textbf{1}_{M\ll Q^2\gcd(M,Q)},
\end{align*}
where $\eta\in \mathcal{X}_0(\Pi_{\infty},\mathfrak{M};\chi_{\infty},\omega),$ and the implied constant depends only on $F,$ $\varepsilon,$ and $c_v,$ $C_v$ at $v\mid\infty,$ cf. \textsection \ref{2.1.2}. 
\end{thmx}

\begin{proof}[Proof of Theorem \ref{A}]
If $C_{\fin}(\chi)>1,$ then Theorem \ref{A} follows from \eqref{8.2}. In the case where $C_{\fin}(\chi)=1$, we replace $\chi$ with $\chi\chi_0$, where $\chi_0$ is a fixed Hecke character induced from a Dirichlet character with a fixed modulus, such as $3$. Similarly, we replace $\pi$ with $\pi\otimes\overline{\chi}_0$. By applying Theorem \ref{thmE} to $\pi\otimes\overline{\chi}_0$ and $\chi\chi_0$, we obtain the same bound (with a different implied constant dependent on the modulus of $\chi_0$) for the second term $L(1/2,\pi\times\chi)$. Consequently, Theorem \ref{A} follows.
\end{proof}

\begin{proof}[Proof of Corollaries \ref{C}]
Let $\pi=\eta\boxplus\eta.$ Then $\omega=\eta^2.$ By \cite[Lemma 3.6]{Yan23c} there exists $t_0\in [2^{-1}\exp(-3\sqrt{\log C(\eta\chi)}), \exp(-3\sqrt{\log C(\eta\chi)})]$ (which might depend on the character $\eta\chi$) such that
\begin{equation}\label{8.3}
|L(1/2,\eta\chi)|\ll \exp(\log^{\frac{3}{4}}C(\eta\chi))|L(1/2+it_0,\eta\chi)|,
\end{equation}
where the implied constant depends only on $F.$ Here $C(\chi\eta)\ll M^{1/2}Q$ is the analytic conductor of $\eta\chi.$ By \cite[(1.2.3)]{Ram95} we have
\begin{align*}
\int_{\mathbb{R}}\frac{|L(1/2+it,\eta\chi)L(1/2+it,\omega\overline{\eta}\chi)|^2}{|L(1+2it,\omega\overline{\eta}^2)|^2}dt\gg \frac{|L(1/2+it_0,\eta\chi)L(1/2+it_0,\omega\overline{\eta}\chi)|^2}{C(\eta\chi)^{\varepsilon}}.
\end{align*}

Since $\omega=\eta^2,$ then $|L(1/2+it_0,\omega\overline{\eta}\chi)|=|L(1/2+it_0,\eta\chi)|.$ So it follows from Theorem \ref{thmE} that, for $\chi\in \mathcal{X}_0(\Pi_{\infty},\mathfrak{M};\chi_{\infty},\omega),$  
\begin{equation}\label{8.4}
|L(1/2+it_0,\eta\chi)|^4\ll T^{1+\varepsilon}M^{1+\varepsilon}Q^{\varepsilon}+T^{\frac{1}{2}+\varepsilon}M^{\varepsilon}Q^{1+\varepsilon}\cdot \textbf{1}_{M\ll Q^2\gcd(M,Q)}.
\end{equation}

Suppose $\eta$ is primitive. Then $C_{\fin}(\eta)=M^{1/2}.$ It then follows from \eqref{8.3} and \eqref{8.4} that 
\begin{align*}
L(1/2,\eta\chi)\ll_{\eta_{\infty},\chi_{\infty}} C_{\fin}(\eta)^{\frac{1}{2}+\varepsilon}+C_{\fin}(\chi)^{\frac{1}{4}+\varepsilon}. 
\end{align*}

By symmetry we also have
\begin{align*}
L(1/2,\eta\chi)\ll_{\eta_{\infty},\chi_{\infty}} C_{\fin}(\eta)^{\frac{1}{4}+\varepsilon}+C_{\fin}(\chi)^{\frac{1}{2}+\varepsilon}.
\end{align*}

Hence the estimate \eqref{1.2.} holds.
\end{proof}

\subsection{Proof of Corollary \ref{cornon}}
Let $f\in\mathcal{F}_{2k}^{new}(N).$ By Hecke's theorem there exists a primitive quadratic character $\chi$ of conductor $q \ll kN^{1+\varepsilon}$ such that $L(1/2,f\times\chi)\neq 0.$ Here the implied constant is absolute.

Let $k\in \{2,3,4,5,7\}.$ Denote by $\mathcal{N}:=\#\big\{g\in \mathcal{F}_{2k}^{\new}(N):\ L(1/2,f\times\chi)L(1/2,g\times\chi)\neq 0\}.$ Recall that (e.g., cf. \cite[p.987]{Kum20})
\begin{align*}
\sum_{g\in \mathcal{F}_{2k}^{\new}(N)}L(1/2,g\times\chi)\gg N^{1-\varepsilon}.
\end{align*}

Then by Cauchy-Schwarz inequality and Corollary \ref{cor1.5} we obtain 
\begin{align*}
N^{1-\varepsilon}\ll \mathcal{N}^{\frac{1}{2}}\cdot \Bigg[\sum_{g\in \mathcal{F}_{2k}^{\new}(N)}|L(1/2,g\times\chi)|^2\Bigg]^{\frac{1}{2}}\ll \mathcal{N}^{\frac{1}{2}}\cdot N^{\frac{1}{2}+\varepsilon},
\end{align*}
leading to \eqref{eq1.6}. Here the implied constant depends only on $\varepsilon.$ 

\begin{remark}
In the above proof, a crucial new ingredient is our Corollary \ref{cor1.5}, which effectively replaces the third moment estimate employed in \cite[Theorem 1]{PY19}:
\begin{equation}\label{8.5}
\sum_{g\in \mathcal{F}_{2k}^{\new}(N)}L(1/2,g\times\chi)^3\ll_{k,\varepsilon}(Nq)^{1+\varepsilon}.
\end{equation}

It is worth noting that Corollary \ref{cor1.5}, given by
\begin{align*}
\sum_{g\in \mathcal{F}_{k}^{\new}(N)}|L(1/2,g\times\chi)|^2\ll (kNq)^{\varepsilon}(kN+k^{\frac{1}{2}}q\cdot \textbf{1}_{N\ll q^2\gcd(N,q)}),
\end{align*}
provides the average Lindel\"{o}f estimate in the $N$-aspect when $q\ll_{k} N^{1+\varepsilon}$. However, \eqref{8.5} does not yield this bound when $q$ is large. 

\end{remark}

\bibliographystyle{alpha}

\bibliography{LY}

\begin{thebibliography}{Yan23b}

\bibitem[Art79]{Art79}
James Arthur.
\newblock Eisenstein series and the trace formula.
\newblock In {\em Automorphic forms, representations and {$L$}-functions
  ({P}roc. {S}ympos. {P}ure {M}ath., {O}regon {S}tate {U}niv., {C}orvallis,
  {O}re., 1977), {P}art 1}, Proc. Sympos. Pure Math., XXXIII, pages 253--274.
  Amer. Math. Soc., Providence, R.I., 1979.

\bibitem[CI00]{CI00}
J.~B. Conrey and H.~Iwaniec.
\newblock The cubic moment of central values of automorphic {$L$}-functions.
\newblock {\em Ann. of Math. (2)}, 151(3):1175--1216, 2000.

\bibitem[FW09]{FW09}
Brooke Feigon and David Whitehouse.
\newblock Averages of central {$L$}-values of {H}ilbert modular forms with an
  application to subconvexity.
\newblock {\em Duke Math. J.}, 149(2):347--410, 2009.

\bibitem[HL94]{HL94}
Jeffrey Hoffstein and Paul Lockhart.
\newblock Coefficients of maass forms and the siegel zero.
\newblock {\em Annals of Mathematics}, pages 161--176, 1994.

\bibitem[HN18]{HP18}
Yueke Hu and Paul~D Nelson.
\newblock New test vector for waldspurger's period integral, relative trace
  formula, and hybrid subconvexity bounds.
\newblock {\em arXiv preprint arXiv:1810.11564}, 2018.

\bibitem[Jac09]{Jac09}
Herv\'{e} Jacquet.
\newblock Archimedean {R}ankin-{S}elberg integrals.
\newblock In {\em Automorphic forms and {$L$}-functions {II}. {L}ocal aspects},
  volume 489 of {\em Contemp. Math.}, pages 57--172. Amer. Math. Soc.,
  Providence, RI, 2009.

\bibitem[Kha21]{Kha21}
Rizwanur Khan.
\newblock Subconvexity bounds for twisted {$L$}-functions.
\newblock {\em Q. J. Math.}, 72(3):1133--1145, 2021.

\bibitem[Kum20]{Kum20}
Balesh Kumar.
\newblock Level aspect of simultaneous nonvanishing for the central
  {$L$}-values associated to cusp forms.
\newblock {\em Proc. Amer. Math. Soc.}, 148(3):979--991, 2020.

\bibitem[Luo17]{Luo17}
Wenzhi Luo.
\newblock On simultaneous nonvanishing of the central {$L$}-values.
\newblock {\em Proc. Amer. Math. Soc.}, 145(10):4227--4231, 2017.

\bibitem[MR12]{MR12}
Philippe Michel and Dinakar Ramakrishnan.
\newblock Consequences of the {G}ross-{Z}agier formulae: stability of average
  {$L$}-values, subconvexity, and non-vanishing mod {$p$}.
\newblock In {\em Number theory, analysis and geometry}, pages 437--459.
  Springer, New York, 2012.

\bibitem[Nel20]{Nel20}
Paul~D Nelson.
\newblock Spectral aspect subconvex bounds for {$\mathrm{U}_{n+1}\times
  \mathrm{U}_n$}.
\newblock {\em arXiv preprint arXiv:2012.02187}, 2020.

\bibitem[Nel21]{Nel21}
Paul~D Nelson.
\newblock Bounds for standard {$L$}-functions.
\newblock {\em arXiv preprint arXiv:2109.15230}, 2021.

\bibitem[NV21]{NV21}
Paul~D Nelson and Akshay Venkatesh.
\newblock The orbit method and analysis of automorphic forms.
\newblock {\em Acta Math.}, (226(1)):1--209, 2021.

\bibitem[PY19]{PY19}
Ian Petrow and Matthew~P. Young.
\newblock A generalized cubic moment and the {P}etersson formula for newforms.
\newblock {\em Math. Ann.}, 373(1-2):287--353, 2019.

\bibitem[PY20]{PY20}
Ian Petrow and Matthew~P. Young.
\newblock The {W}eyl bound for {D}irichlet {$L$}-functions of cube-free
  conductor.
\newblock {\em Ann. of Math. (2)}, 192(2):437--486, 2020.

\bibitem[Ram95]{Ram95}
Kanakanahalli Ramachandra.
\newblock {\em On the mean-value and omega-theorems for the Riemann
  zeta-function}.
\newblock Number~85. Tata Institute of Fundamental Research, 1995.

\bibitem[RR05]{RR05}
Dinakar Ramakrishnan and Jonathan Rogawski.
\newblock Average values of modular {$L$}-series via the relative trace
  formula.
\newblock {\em Pure Appl. Math. Q.}, 1(4, Special Issue: In memory of Armand
  Borel. Part 3):701--735, 2005.

\bibitem[Yan23a]{Yan23a}
Liyang Yang.
\newblock Relative trace formula and {$L$}-functions for
  {$\mathrm{GL}(n+1)\times \mathrm{GL}(n)$}.
\newblock {\em arXiv preprint arXiv:2303.02225}, 2023.

\bibitem[Yan23b]{Yan23c}
Liyang Yang.
\newblock Relative trace formula and twisted $ l $-functions: the burgess
  bound.
\newblock {\em arXiv preprint arXiv:2305.10719}, 2023.

\bibitem[You17]{You17}
Matthew~P. Young.
\newblock Weyl-type hybrid subconvexity bounds for twisted {$L$}-functions and
  {H}eegner points on shrinking sets.
\newblock {\em J. Eur. Math. Soc. (JEMS)}, 19(5):1545--1576, 2017.

\end{thebibliography}

\end{document}